\documentclass[a4paper,11pt]{article}
\usepackage{amsmath}
\usepackage{amssymb}
\usepackage{mathrsfs}
\usepackage{amsfonts}
\usepackage{amsthm}
\usepackage{pstricks, pst-node, pst-text, pst-3d,psfrag}
\usepackage{graphicx}
\usepackage{indentfirst}
\usepackage{enumerate}
\usepackage{subfigure}
\usepackage{cite}
\usepackage[colorlinks=true]{hyperref}
\hypersetup{urlcolor=blue, citecolor=red}

\theoremstyle{plain}
\newtheorem{thm}{Theorem}[section]
\newtheorem{thmm}{Theorem}[subsection]
\newtheorem{prop}{Proposition}[section]
\newtheorem{cor}{Corollary}[section]
\newtheorem{lem}{Lemma}[section]
\theoremstyle{definition}
\newtheorem{defn}{Definition}[section]

\newtheorem{rmk}{\textbf{Remark}}[section]

\numberwithin{equation}{section}

\makeatletter 
\@addtoreset{equation}{section}
\makeatother  

\newcommand\norm[1]{\lVert#1\rVert}
\newcommand\abs[1]{\lvert#1\rvert}

\begin{document}

\title{Generic behavior of differentially positive systems\\ on a globally
orderable Riemannian manifold}

\setlength{\baselineskip}{16pt}


\author{
Lin Niu\thanks{Supported by the National Natural Science Foundation of China No.12201034 and 12090012.}\\
School of Mathematics and Physics\\
University of Science and Technology Beijing\\
Beijing, 100083, P. R. China
\\[3mm]
Yi Wang\thanks{Supported by the National Natural Science Foundation of China No.12331006.}
\\
School of Mathematical Sciences\\
University of Science and Technology of China\\
Hefei, Anhui, 230026, P. R. China
}

\date{}
\maketitle

\begin{abstract}

Differentially positive systems are the nonlinear systems whose linearization along trajectories preserves a cone field on a smooth Riemannian manifold. One of the embryonic forms for cone fields in reality is originated from the general relativity. By utilizing the Perron-Frobenius vector fields and the $\Gamma$-invariance of cone fields, we show that generic (i.e.,``almost all" in the topological sense) orbits are convergent to certain single equilibrium. This solved a reduced version of Forni-Sepulchre's conjecture in 2016 for globally orderable manifolds.
\vskip 3mm

\end{abstract}


\section{Introduction}

Differential analysis provides a general framework for investigating a nonlinear dynamical system by analyzing the linearization of the system at every point in the state space. The
motivation is that the local behavior of a system often has strong implications for the global nonlinear behavior. The current series of papers focus on the nonlinear dynamical system whose linearization along trajectories preserves a cone field. Here, a cone field $C_M$ assigns to each point $x$ of a manifold $M$ a closed convex cone $C_M(x)$ in the tangent space $T_{x}M$ of $x$.

One of the embryonic forms for cone fields in reality is originated from the general relativity, for which the time-orientable space-time (i.e., a connected 4-dimensional Hausdorff $C^{\infty}$-manifold  endowed with a Lorentz metric) naturally generates at each point a Lorentzian cone in tangent space (see, e.g. \cite{Beemandehrlich81,Hawkingandellis73,Penrose72} or Appendix A.1). The structure of a cone field in the time-orientable space-time turns out to be one of the crucial mathematical tools in the study of general relativity including causality theory, singularity theory and black holes, etc.. Among others, the well-known non-spacelike curve (also called causal curve) in these theories (c.f. Hawking and Ellis \cite{Hawkingandellis73} and Penrose \cite{Penrose72}) is actually one whose tangent vector at each point falls in the Lorentzian cone field. Moreover, the causally related points in a space-time should be joined by a non-spacelike curve.

From this point of view, a cone field on the manifold $M$ naturally induces a ``conal order relation" as follows: two points $x,y\in M$ are {\it conal ordered}, denoted by $x\leq_M y$, if there exists a conal curve on $M$ beginning at $x$ and ending at $y$. Here, a conal curve is a piecewise smooth curve whose tangent vector lies in the cone at every point along the curve wherever it is defined (see Definition \ref{conal curve} and Fig.\ref{figa}). In particular, in the setting of space-times (see \cite{Segal76,GutandLevichev,Levichev} and \cite{HilgertHofmannLawson,Olshanskii82,Viberg}), the causal curves are actually the conal curves; while, the causally related points are the ones that are conal ordered. Besides, conal order also has various applications in hyperbolic partial differential equations and harmonic analysis (see \cite{Faraut87,Faraut91}), as well as in the theory of Wiener-Hopf operators on symmetric spaces \cite{HilgertNeeb}.

It deserves to point out that, unlike the standard order relation induced by a single closed convex cone in a topological vector space, the conal order relation ``$\leq_M$" induced by a cone field on $M$ is certainly reflexive and transitive, but not always antisymmetric. In fact, there are examples of $M$ that contain the closed conal curves (e.g., the closed timelike curves in time-orientable space-times \cite[Chapter 5]{Hawkingandellis73}), which reveals that the anti-symmetry fails.
Moreover, the relation ``$\leq_M$" is not necessarily closed, i.e., the set $\{(x,y)\in M\times M: x\leq_M y \}$ is not a closed subset in $M\times M$ (see \cite[p.299]{Lawson89} or \cite[p.470]{Neeb91}). For instance, such set is not necessarily closed in Minkowski space (see \cite[p.183]{Hawkingandellis73} or \cite[p.12]{Penrose72}).

In many recent works \cite{Forni2015,ForniandSepulchre14,ForniandSepulchre16,MostajeranSepulchre18SIAM,MostajeranSepulchre18homogeneous}, the nonlinear system whose linearization along trajectories preserves a convex cone field is also referred as differentially positive system, the flow of which naturally keeps the conal order ``$\leq_M$". To the best of our knowledge, Forni and Sepulchre \cite{Forni2015,ForniandSepulchre14,ForniandSepulchre16} first studied the dynamics of differentially positive systems. Among others, they constructed on $M$ a canonical defined vector field $\mathcal{W}$, called the Perron-Frobenius vector field, such that the vector $\mathcal{W}(x)$ lines in the interior of the cone $C_M(x)$ at each point $x\in M$. By appealing to the Perron-Frobenius vector field $\mathcal{W}$, a dichotomy characterization of limit sets for differentially positive systems was provided. Based on this, Forni and Sepulchre \cite[p.353]{ForniandSepulchre16} further posed the following
\vskip 3mm

\noindent {\bf Conjecture} ([Forni and Sepulchre \cite{ForniandSepulchre16}). {\it For almost every $x\in M$, the $\omega$-limit set $\omega(x)$ is given by either a fixed point, or a limit cycle, or fixed points and connecting arcs.
}

\vskip 3mm

This conjecture predicts what typically happens to orbits, as time increases to $+\infty$. The description of typical properties of orbits usually means those shared by ``most" orbits.

\vskip 2mm
In our present work, we will tackle this conjecture and made an attempt to establish the asymptotic behavior of ``most" orbits for differentially positive systems. For this purpose, we first introduce the following assumption for $M$:

\begin{enumerate}[{\bf (H1)}]
	\item $M$ is globally orderable equipped with a continuous solid cone field.
\end{enumerate}

(H1) means that the conal order ``$\le_M$" on $M$ is a partial order (see \cite{Lawson89,HilgertHofmannLawson,HilgertNeeb93,Neeb91} and Definition \ref{globally orderable}); and hence, it is actually anti-symmetric.
This occurs naturally in various situations. One of the well-known examples is a homogeneous space of positive definite matrices with the affine-invariant cone field (see, e.g, \cite[Section 3]{MostajeranSepulchre18orderingmatrices}).
As a matter of fact, (H1) excludes the occurrence of the closed conal curves (see \cite[Section 5]{Lawson89}). In the setting of space-times, such closed conal (causal) curves would seem to generate paradoxes involving causality and are said to ``violate causality" (\cite{Beemandehrlich81,Hawkingandellis73,Hawkingandsachs74,Levichev,Penrose72}).
\vskip 2mm

As a consequence, by (H1), the possibility of a limit cycle, or fixed points and connecting arcs (which are all closed conal curves) can be ruled out. In other words, Forni-Sepulchre's Conjecture is naturally reduced to the following

\vskip 3mm
\noindent {\bf Conjecture A}. {\it For almost every $x\in M$, the $\omega$-limit set $\omega(x)$ is a singleton.
}
\vskip 3mm

In the present paper, we will prove conjecture A under the following two reasonable assumptions:

\begin{enumerate}
	\item[{\bf (H2)}] The conal order ``$\leq_M$" is quasi-closed.
    \item[{\bf (H3)}] Both the cone field $C_M$ and the Riemannian metric on $M$ are $\Gamma$-invariant.
\end{enumerate}

More precisely, our main result is the following

\begin{thmm}\label{genericconvergence}
	Assume that {\rm (H1)-(H3)} hold. Then, for almost every $x\in M$, the $\omega$-limit set $\omega(x)$ is a singleton.
\end{thmm}

This theorem, in a detailed version, will be proved in Section 3 (see Theorem \ref{convergent thm}).
It concludes that generic (i.e.,``almost all" in the topological sense) orbits are convergent to a single equilibrium.

The quasi-closedness of ``$\leq_M$" in (H2) means that {\it $x\leq_M y$ whenever $x_n\to x$ and $y_n\to y$ as $n\to \infty$ and $x_n\ll_M y_n$ for all integer $n\ge 0$} (see Definition \ref{quasi-closed}). Here, we write $x\ll_M y$ if there exists a so-called \textit{strictly} conal curve (whose tangent vector lies in the interior of cone at every point along the curve) on $M$ beginning at $x$ and ending at $y$ (see Definition \ref{conal curve}). We point out that (H2) is motivated by the causal continuity of space-times in Hawking and Sachs  \cite{Hawkingandsachs74}, while the notation $``\ll_M"$ is derived from the timelike curve in general relativity (see e.g., \cite{Penrose72}). In fact, one can prove that the conal order $``\le_M"$ is indeed quasi-closed in many causally continuous space-times (see Theorem \ref{quasiclosed} in the Appendix).

As for (H3), a cone field $C_M$ is called \textit{$\Gamma$-invariant} if there exists a linear invertible mapping $\Gamma(x_1,x_2) : T_{x_1}M \to T_{x_2}M$ for all $x_1,x_2\in M$, such that $\Gamma(x_1,x_2)C_M(x_1)=C_M(x_2)$.
While, the Riemannian metric $(\cdot,\cdot)_x$ on $M$ is \textit{$\Gamma$-invariant} if $(u,v)_{x_1}=(\Gamma(x_1,x_2)u,\Gamma(x_1,x_2)v)_{x_2}$ for all $x_1,x_2\in M$ and $u,v\in T_{x_1}M$.
Actually, $\Gamma$-invariance is motivated by the homogeneous structure of the manifolds.
A well-known example admitting (H3) is a homogeneous space $M=G/H$ of a connected Lie group $G$ assigned to each point $x\in M$ a closed convex cone in the tangent space $T_{x}M$ such that the cone field is invariant under the action of $G$ (see e.g., \cite{Lawson89,Neeb91,HilgertNeeb93}); and
a homogeneous Riemannian metric on such homogeneous space is naturally $\Gamma$-invariant (see e.g., \cite{Arvanitoyeorgos,Helgason}).
In the Appendix A.2, more detailed structures of the globally orderable homogeneous spaces are presented.
We further mention here that the homogeneous cone field comes from the Lie theory.
It is shown that cone fields arise as quotients these so-called Lie wedge fields (see \cite[Lemma VI.1.5]{HilgertHofmannLawson}). This provides a crucial link between the Lie theory of semigroups and that of cone fields on manifolds (see \cite{HilgertHofmannLawson,HilgertNeeb93,HilgertOlafsson97,Lawson89,Neeb91}). Very recently, the works in \cite{Neebolafsson21,Neebolafsson21b,Neebolafsson23,Morinellineeb21,Neeborstedolafsson21}
form an ongoing project aiming at the connections between causal structures on homogeneous spaces, algebraic quantum field theory, modular theory of operator algebras and unitary representations of Lie groups.
\vskip 2mm

Needless to say, differential positivity in Forni and Sepulchre \cite{ForniandSepulchre16} with respect to a constant cone field on a flat space $M$ is reduced to the classical monotonicity (\cite{H88,HS05,S95,S17,Chueshov02,Jiangandzhao05,ShenandYi,Polacik02,Polacik89,SmithThieme91,HessandPolacik,Mierczynskiandshen08,PolacikandTerescak92,Terescak,WangandYao20,WangandYao22}) with respect to a closed convex cone. From this point of view, differentially positive systems can be regarded as a natural generalization of the so-called classical monotone systems to nonlinear manifolds.
We pointed out that the order introduced by a closed convex cone is a closed partial order (see the detail in Appendix A.3). In such a flat space, Theorem \ref{genericconvergence} automatically implies the celebrated Hirsch's Generic Convergence Theorem \cite{H88}.

The paper is organized as follows. In Section $2$, we introduce some notations and definitions and summarize some preliminary results. Theorem \ref{genericconvergence} (i.e., Theorem \ref{convergent thm}) with its proof will be given in Section $3$. Several fundamental tools and critical lemmas, which turns out be useful in the proof of Theorem \ref{genericconvergence} (or Theorem \ref{convergent thm}), will be postponed in Section $4$. Finally, in the Appendix, we will present the order structures on space-times, the globally orderable homogeneous spaces with homogeneous cone fields, as well as the differential positivity in the flat spaces.

\section{Notations and preliminary results}

Throughout this paper, $M$ will be reserved as a smooth manifold of dimension $n$. The tangent bundle is denoted by $TM$ and the tangent space at a point $x\in M$ by $T_{x}M$. $M$ is endowed with a Riemannian metric tensor, represented by a inner product $(\cdot,\cdot)_x$ on the tangent space $T_{x}M$.
We set $\abs{v} _x:=\sqrt{(v,v)_x}$ for any $v\in T_{x}M$ (sometimes we omit the subscripts $x$ in $\abs{\cdot}_x$) and the Riemannian metric endows the manifold with the Riemannian distance $d$. We assume that $(M,d)$ is a complete metric space. 

Let $E$ be a $n$ dimensional real linear space. A nonempty closed subset $C$ of the linear space $E$ is called a \textit{closed convex cone} if $C+C\subset C$, $\alpha C\subset C$ for all $\alpha \geq 0$, and $C \cap (-C) = \{0\}$. A convex cone $C$ is \textit{solid} if its interior $\text{Int} C \neq \emptyset$. A convex cone $C'$ in $E$ is said to \textit{surround} a cone $C$ if $C\setminus \{ 0\}\subset \text{Int}C'$.

A \textit{cone field} on a manifold $M$ is a map $x \mapsto C_{M}(x)$, such that $C_{M}(x)$ is a closed convex cone in $T_{x}M$ for each $x\in M$.
A manifold equipped with a cone field is called a \textit{conal manifold}. A cone field is \textit{solid} if each convex cone $C_{M}(x)$ is solid.

A cone field $C_M$ is called a \textit{$\Gamma$-invariant} cone field if there exists a linear invertible mapping $\Gamma(x_1,x_2) : T_{x_1}M \to T_{x_2}M$ for each $x_1,x_2\in M$, such that $\Gamma(x_1,x_2)C_M(x_1)=C_M(x_2)$. Moreover, $\Gamma$ is continuous with respect to $(x_1, x_2)$ and $\Gamma(x,x)=\text{Id}_x$ for all $x\in M$, where $\text{Id}_x : T_xM \to T_xM$ is the identity map.
We say that the Riemannian metric is \textit{$\Gamma$-invariant} under the linear invertible mapping $\Gamma$ if $(u,v)_{x_1}=(\Gamma(x_1,x_2)u,\Gamma(x_1,x_2)v)_{x_2}$ for all $x_1,x_2\in M$ and $u,v\in T_{x_1}M$. In this case, $\abs{v}_{x_1}=\abs{\Gamma(x_1,x_2) v}_{x_2}$ holds for all $x_1,x_2\in M$ and $v\in T_{x_1}M$.

Pick a smooth chart $\Phi : U \to V$, where $U\subset M$ is an open set containing $x$ and $V\subset \mathbb{R}^n$ is an open set. Let $C_{M}^{\Phi}(y)\subset \mathbb{R}^n$ denote the representation of the cone field in the chart, i.e., $d\Phi(y)(C_{M}(y))=\{ \Phi(y)\}\times C_{M}^{\Phi}(y)$. Following \cite{Lawson89},
a cone field $x \mapsto C_{M}(x)$ on $M$ is \textit{upper semicontinuous} at $x\in M$ if given a smooth chart $\Phi : U \to \mathbb{R}^{n}$ at $x$ and a convex cone $C'$ surrounding $C_{M}^{\Phi}(x)$, there exists a neighborhood $W$ of $x$ such that $C_{M}^{\Phi}(y)\subset C'$ for all $y\in W$. The cone field is \textit{lower semicontinuous} at $x$ if given any open set $N$ such that $N\cap C_{M}^{\Phi}(x)$ is non-empty, there exists a neighborhood $W$ of $x$ such that $N\cap C_{M}^{\Phi}(y)$ is non-empty for all $y\in W$. A cone field is \textit{continuous at $x$} if it is both lower and upper semicontinuous at $x$ and \textit{continuous} if it is continuous at every $x$.

\begin{defn}\label{conal curve}
    A continuous piecewise smooth curve $t \mapsto \gamma(t)$ defined on $[t_0,t_1]$ into a conal manifold $M$ is called a \textit{conal curve} if $\gamma'(t)\in C_{M}(\gamma(t))$ whenever $t_0\leq t<t_1$, in which the derivative is the right hand derivative at those finitely many points where the derivative is not continuous. (See Fig.\ref{figa}). Moreover, $\gamma$ is called a \textit{strictly conal curve} if $\gamma'(t)\in \text{Int} C_{M}(\gamma(t))$ for $t_0\leq t<t_1$.
\end{defn}

For two points $x,y\in M$, we say $x$ and $y$ are ordered, denoted by $x\leq_M y$, if there exists a conal curve $\alpha : [t_0,t_1]\subset \mathbb{R} \to M$ such that $\alpha(t_0)=x$ and $\alpha(t_1)=y$.
This relation is an \textit{order} on $M$. In fact, it is always reflexive (i.e., $x\leq_M x$ for all $x\in M$) and transitive (i.e., $x\leq_M y$ and $y\leq_M z$ implies $x\leq_M z$). The order ``$\leq_M$" is referred to as the \textit{conal order}.
We write $x\ll_M y$ if there exists a so-called \textit{strictly} conal curve $\gamma$ with $\gamma(t_0)=x$ and $\gamma(t_1)=y$. Clearly, the relation ``$\ll_M$" is always transitive.
Let $U,V$ be the subsets of $M$. We write $U\leq_M V$ (resp., $U\ll_M V$) if $x\leq_M y$ (resp., $x\ll_M y$) for any $x\in U$ and $y\in V$.

The following proposition implies that the relation ``$\ll_M$" is \textit{open}.

\begin{prop}\label{strong order open property}
	If $x\ll_M y$, then there exist neighborhoods $U$ of $x$, $V$ of $y$ such that $U\ll_M V$.
\end{prop}
\begin{proof}
We only prove the existence of $V$, i.e., if $x\ll_M y$, there exists a neighborhood $V$ of $y$ such that $x\ll_M z$ for each $z\in V$. Since $x\ll_M y$, there is a strictly conal curve $\gamma_1 : [t_0,t_1]\subset \mathbb{R} \to M$ such that $\gamma_1(t_0)=x$ and $\gamma_1(t_1)=y$. Let $\gamma(t)=\gamma_1(t_1-t)$, $t\in [0,t_1-t_0]$. Then $\gamma(0)=y$, $\gamma(t_1-t_0)=x$ and $\gamma'(0)\in \text{Int} \{-C_{M}(\gamma(0))\}$.
	
	Let $\Phi : U \to \Phi(U) \subset \mathbb{R}^n$ be a chart with $\Phi(\gamma(0))=0$ such that $\Phi(U)$ is convex. We set $-\widetilde{C}(\Phi(z))=d\Phi(z)(-C_M(z))$ for $z\in U$ and $\alpha=\Phi \circ \gamma$. Then $\alpha'(0)\in \text{Int}\{-\widetilde{C}(0)\}$. So, we can choose a compact convex neighborhood $B$ of $\alpha'(0)$ in the interior of $-\widetilde{C}(0)$ and set $W=\mathbb{R}^+ B :=\{\lambda B : \lambda >0\} \subset -\widetilde{C}(0)$. Since the cone field is continuous, there is a neighborhood $V\subset M$ of $\gamma(0)$ such that $W \subset -\widetilde{C}(z)$ for all $z\in U'=\Phi(V)$ and $U'$ is convex.
	
	On the other hand, $\alpha'(0)=\lim\limits_{t\to 0^+}\frac{1}{t}\alpha (t)\in \text{Int}W$. Then there is a $s_0>0$ such that $\alpha((0,s_0])\subset \text{Int}W$ and $\alpha([0,s_0])\subset U'$. Let $B'$ be an open convex neighborhood of $\alpha(s_0)$ in $W\cap U'$.
	
	Let $B''=\{\alpha(s_0)-u : u\in B'\}\cap U'$, then $B''$ is an open neighborhood of $\alpha(0)$. For each $v\in B''$, we set $\alpha_v(t)=\alpha(s_0)+t(v-\alpha(s_0))$, $t\in [0,1]$. Then $\alpha_v(0)=\alpha(s_0)$, $\alpha_v(1)=v$ and $\alpha_{v}^{'}(t)=v-\alpha(s_0)\in -B' \subset \widetilde{C}(\alpha_v(t))$.
	
	Thus, there is a point $p\in \gamma_1$ and an open neighborhood $D$ of $y$ such that $p\ll_M z$ for all $z\in D$. Since $x\ll_M p$, we obtain $x\ll_M z$ for all $z\in D$.
\end{proof}

According to our standing assumption (H1) in the introduction, we now give the following crucial definitions:

\begin{defn}\label{globally orderable}
	A conal manifold is said to be \textit{globally orderable} if the order ``$\leq_M$" is a partial order which is locally order convex.
\end{defn}

The conal order ``$\leq_M$" is a \textit{partial} order relation if it is additionally antisymmetric (i.e., $x\leq_M y$ and $y\leq_M x$ implies $x=y$).
A partial order ``$\leq$" is \textit{locally order convex} if it has a basis of neighborhoods at each point that are order convex ($z, x\in U$ implies $\{y: z\leq y \leq z\}\subset U$). See \cite{Lawson89,Neeb91}.

\begin{rmk}
In the terminology of general relativity, each point $p$ in a globally orderable manifold $M$ is called a strong point (see Lawson \cite[Section 5]{Lawson89}), by which means that every neighborhood $U$ of $p$ contains a smaller neighborhood $V$ of $p$ such that every conal curve that begins in $V$ and leaves $U$ terminates outside of $V$ (see e.g., \cite[p.192]{Hawkingandellis73}, \cite[p.28]{Penrose72} or \cite[p.59]{Beemandehrlich81}).
\end{rmk}

\begin{rmk}
	The {\it globality of the conal order}, i.e., whether the conal order associated with a cone field can be extended to a partial order on the global manifold, is a central problem in \cite{Lawson89}. The equivalence between globality of the conal order in a homogeneous manifold and globality of the Lie wedge in the Lie group has been shown in \cite[Section 5]{Lawson89} (see also \cite[Theorem 1.6]{Neeb91} and \cite[Section 4.3]{HilgertNeeb93}).
	Globality of a conal order occurs naturally in various situations; for example,	the conal order induced by the affine-invariant cone field on the homogeneous space of positive definite matrices, as we mentioned above, is a partial order (see \cite[Theorem 2]{MostajeranSepulchre18orderingmatrices}).
	Besides, if the manifold is globally orderable, one can exclude the occurrence of the closed conal curves (see \cite[Section 5]{Lawson89}).
\end{rmk}

\begin{defn}\label{quasi-closed}
	The order ``$\leq_M$" is \textit{quasi-closed} if $x\leq_M y$ whenever $x_n\to x$ and $y_n\to y$ as $n\to \infty$ and $x_n\ll_M y_n$ for all $n$.
\end{defn}

\begin{rmk}
	In general, the conal order is not a closed order ($x\leq_M y$ whenever $x_n\to x$ and $y_n\to y$ as $n\to \infty$ and $x_n\leq_M y_n$), see \cite[p.299]{Lawson89} and \cite[p.470]{Neeb91}. The quasi-closed order relationship here is inspired by the causal continuity of space-times studied by Hawking and Sachs in \cite{Hawkingandsachs74}.
	One may refer to Appendix A.1 for more details.
\end{rmk}

\vskip 3mm

Now, we consider a system $\Sigma$ generated by a smooth vector field $f$ on $M$ equipped with a continuous solid cone field $C_M$. The induced flow by system $\Sigma$ is denoted by $\varphi_t$.
We write $d\varphi_t(x)$ as the tangent map from $T_{x}M$ to $T_{\varphi_{t}(x)}M$. Let $x\in M$, the \textit{positive semiorbit} (resp., \textit{negative semiorbit}) of $x$, be denoted by $O^+(x)=\{\varphi_t(x):t\geq 0\}$ (resp.,  $O^-(x)=\{\varphi_t(x):t\leq 0\}$). The \textit{full orbit} of $x$ is denoted by $O(x)=O^+(x)\cup O^-(x)$. An \textit{equilibrium} is a point $x$ such that $O(x)=\{x\}$. Let $E$ be the set of all the equilibria of $\varphi_t$. The $\omega$-limit set $\omega(x)$ of $x$ is defined by $\omega(x)=\cap_{s\geq 0}\overline{\cup_{t\geq s}\varphi_t(x)}$. A point $z\in \omega(x)$ if and only if there exists a sequence $\{t_i\}$, $t_i\to \infty$, such that $\varphi_{t_i}(x)\to z$ as $i\to \infty$. If $O^+(x)$ is precompact, then $\omega(x)$ is nonempty, compact, connected, and invariant (i.e., $\varphi_t(\omega(x))=\omega(x)$ for any $t\in  \mathbb{R}$).
A point $x$ is called a \textit{convergent point} if $\omega(x)$ consists of a single point of $E$. The set of all convergent points is denoted by $C$.

Throughout the paper, we always assume that all orbits are forward complete, which means $\varphi_t(x)$ is well defined for all $t\geq 0$, and the orbit of $x$ has compact closure for each $x\in M$.


\begin{defn}\label{differential positive}
The system $\Sigma$ is said to be \textit{differentially positive} (DP) with respect to $C_{M}$ if
\begin{equation*}
d\varphi_t(x)C_{M}(x)\subseteq C_{M}(\varphi_t(x)), \ \ \forall x\in M, \ \ \forall t\geq 0.
\end{equation*}
And the differentially positive system $\Sigma$ is said to be \textit{strongly differentially positive} (SDP) with respect to $C_{M}$ if
\begin{equation*}
d\varphi_t(x)\{ C_{M}(x)\backslash \{ 0\} \} \subseteq \text{Int} C_{M}(\varphi_t(x)), \ \ \forall x\in M, \ \ \forall t>0.
\end{equation*}
See Fig.\ref{figb}.
\end{defn}
In the present paper, we focus on the strongly differentially positive system $\Sigma$ on a Riemannian manifold $M$.
We first give the following two useful preliminary results.

\begin{prop}\label{monotone property}
If $x\leq_M y$, then $\varphi_t(x)\leq_M \varphi_t(y)$ for $t\geq0$. Moreover, $\varphi_t(x)\ll_M \varphi_t(y)$ for all $t>0$.
\end{prop}
\begin{proof}
	If $x\leq_M y$, there exists a conal curve $\gamma(s)$ such that $\gamma(0)=x$ and $\gamma(1)=y$. Since $\Sigma$ is SDP and $\frac{d}{ds}\gamma(s)\in C_{M}(\gamma(s))\backslash \{ 0\}$, then $\frac{d}{ds}\varphi_t(\gamma(s))=d\varphi_t(\gamma(s))\frac{d}{ds}\gamma(s)\in \text{Int} C_{M}(\varphi_t(\gamma(s)))$ for $t>0$. So, $\varphi_t(\gamma(s))$ is a strictly conal curve.
\end{proof}

\begin{prop}\label{strong flow open relation}
	If $x\leq_M y$, then for $t_0>0$, there exist neighborhoods $U$ of $x$, $V$ of $y$ such that $\varphi_{t}(U)\ll_M \varphi_{t}(V)$ for any $t\geq t_0$.
\end{prop}
\begin{proof}
	Since $x\leq_M y$, $\varphi_{t_0}(x)\ll_M \varphi_{t_0}(y)$ for $t_0>0$. One can take neighborhoods $\bar{U}$ of $\varphi_{t_0}(x)$ and $\bar{V}$ of $\varphi_{t_0}(y)$ such that $\bar{U}\ll_M\bar{V}$ by Proposition \ref{strong order open property}. By the continuity of $\varphi_{t_0}$, there are neighborhoods $U$ of $x$, $V$ of $y$ such that $\varphi_{t_0}(U)\subset \bar{U}$ and $\varphi_{t_0}(V)\subset \bar{V}$.
\end{proof}

\begin{figure}[!h]
	\centering
	\subfigure[\scriptsize  A conal curve $\gamma$ on manifold $M$.]{
		\includegraphics[width=7.3cm]{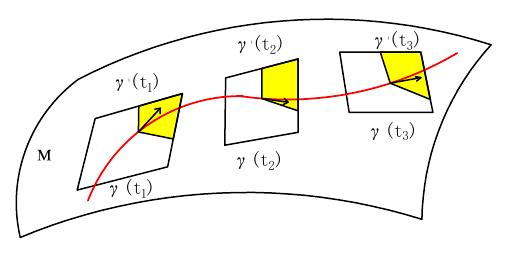}\label{figa}
	}
	\quad
	\subfigure[\scriptsize Strongly differentially positive system on manifold $M$]{
		\includegraphics[width=7.3cm]{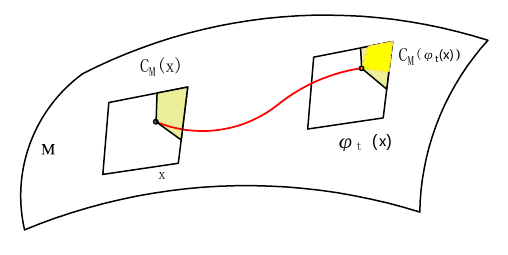}\label{figb}
	}
	\center{{\bf Fig 1}: Conal curves and strongly differentially positive flows on manifold $M$.}
\end{figure}

 As mentioned in the introduction, we hereafter impose the following hypotheses:
\begin{enumerate}[{\bf (H1)}]
	\item $M$ is a globally orderable conal manifold equipped with a continuous solid cone field $C_M$.
	\item The conal order ``$\leq_M$" is quasi-closed.
	\item Both the cone field $C_M$ and the Riemannian metric on $M$ are $\Gamma$-invariant.
\end{enumerate}

Before ending this section, we give the following critical lemma, which turns out to be important for the proof of our main results in the forthcoming sections.

\begin{lem}\label{fundamentalresults}
 Assume that {\rm (H1)-(H2)} hold. Then
	\begin{enumerate}[{\rm (a)}]
		
		\item The $\omega$-limit set cannot contain two points $x$ and $y$ with $x\leq_M y$;
		
		\item If $x\leq_M y$, then either $\omega(x)\ll_M \omega(y)$, or $\omega(x)=\omega(y)\subset E$.
	\end{enumerate}
\end{lem}

For the sake of the completeness, we will postpone its proof in Section 4.

\section{Proof of Theorem \ref{genericconvergence}}

In this section, we will focus on the generic behavior of the SDP flow $\varphi_t$ under the hypotheses (H1)-(H3).

\begin{thm}\label{convergent thm}
	{\rm (Generic Convergence)} Assume that {\rm(H1)-(H3)} hold.
	If the cone field admits a $C^{\infty}$-section and $\varphi_t$ satisfies the compactness condition {\rm (P)}, then ${\rm Int} C$ is dense in $M$.
\end{thm}

Recall that $C$ denotes the set of all convergent points.
Here, we say that a cone field admits a \textit{$C^{\infty}$-section} if there is a $C^{\infty}$ vector field $X$ such that $X(x)\in C_M(x)$ for all $x\in M$.
We also formulate the \textit{compactness condition {\rm (P)}}:
For each $x_0\in M$, $\cup_{n\geq 0}\omega(x_n)$ has compact closure contained in $M$, where $\{x_n\}\to x_0$ with $x_n\leq_M x_{n+1}\leq_M x_0$ (or $x_0\leq_M x_{n+1}\leq_M x_n$) for $n\geq 1$.

\begin{rmk}
We point out that the homogeneous cone field on a homogeneous space naturally admits $C^{\infty}$-sections (see \cite{Lawson89} or Proposition \ref{homogeneousconecontinuous} in the Appendix).
While the compactness condition {\rm (P)} is a relatively weak compactness requirement which is frequently satisfied (see e.g., \cite{S95,S17}).
\end{rmk}

The proof of Theorem \ref{convergent thm} will be divided into several steps. To proceed it, we recall the so-called \textit{Perron-Frobenius vector field}, which was first introduced by Forni and Sepulchre \cite{ForniandSepulchre16}. A continuous vector field $v$ on an invariant compact set $\Omega\subset M$ is called a \textit{Perron-Frobenius vector field} on $\Omega$,
if for each $x\in \Omega$, \vskip 1mm

(i) $v(x)\in \text{Int} C_M(x)$ with $\vert v(x)\vert_x=1$; \vskip 1mm

(ii) $v(\varphi_t(x))=\dfrac{d\varphi_t(x)v(x)}{\abs{d\varphi_t(x)v(x)}}$ for $t\geq 0$;\vskip 1mm

(iii) $\lim\limits_{t\to \infty}d_{\varphi_t(x)}(d\varphi_t(x)u,v(\varphi_t(x)))=0$ for all $u\in C_M(x)\backslash \{ 0\}$, \vskip 1mm

\noindent where $d_{\varphi_t(x)}(\cdot,\cdot)$ is the Hilbert Metric induced by $C_M(\varphi_t(x))$ (see \cite[Section VI]{ForniandSepulchre16}).

\begin{lem}\label{PerronFrobenius}
Let $\omega(x)\subset E$ and $v$ be the Perron-Frobenius vector field on $\omega(x)$. If $\omega(x)$ is not a singleton, then
there exist $\tau>0$ and $\rho(e)>1$ such that $$d\varphi_{\tau}(e)v(e)=\rho(e) v(e),\quad \textnormal{ for }e\in \omega(x).$$ Moreover, there exists $\rho>1$ such that $\rho(e)>\rho$ for all $e\in \omega(x)$.
\end{lem}
\begin{proof}
Fix $\tau>0$. Since $e\in\omega(x)\subset E$, then $\varphi_t(e)=\varphi_{\tau}(e)=e$ for all $t\geq 0$. So, $d\varphi_{\tau}(e)v(e)=\rho(e) v(e)$ for some $\rho(e) >0$.

We then pick a smooth chart $\Phi : U \to W$, where $U\subset M$ is the coordinate neighborhood of $e$, $W$ is an open set in $\mathbb{R}^n$. Using this coordinate map $\Phi$, we write $\widetilde{e}=\Phi(e)$, $\widetilde{v}(\widetilde{e})=d\Phi(e)v(e)$ and $\widetilde{C}(\widetilde{e})=d\Phi(e)C_M(e)$. Let $A = \Phi \circ \varphi_{\tau} \circ \Phi^{-1}$. Then $A(\widetilde{e})=\widetilde{e}$ and $dA(\widetilde{e})=d\Phi(e) \circ d\varphi_{\tau}(e) \circ d\Phi^{-1}(\widetilde{e}) : T_{\widetilde{e}}\mathbb{R}^n \to T_{\widetilde{e}}\mathbb{R}^n$. So, $dA(\widetilde{e})\widetilde{v}(\widetilde{e})=d\Phi(e) \circ d\varphi_{\tau}(e) \circ d\Phi^{-1}(\widetilde{e})\widetilde{v}(\widetilde{e})=d\Phi(e) \circ d\varphi_{\tau}(e) v(e)=d\Phi(e) \rho(e) v(e)=\rho(e) d\Phi(e) v(e)=\rho(e) \widetilde{v}(\widetilde{e})$. Since system is SDP and $\widetilde{C}(\widetilde{e})=d\Phi(e)C_M(e)$, then $dA(\widetilde{e})(\widetilde{C}(\widetilde{e})\backslash \{ 0\}) \subset \text{Int} \widetilde{C}(\widetilde{e})$. Moreover, $\widetilde{v}(\widetilde{e})\in \text{Int}\widetilde{C}(\widetilde{e})$, since $v(e)\in \text{Int}C_M(e)$. By the Perron-Frobenius Theorem, we obtain that $\rho(e) = \sigma(dA(\widetilde{e}))$, where $\sigma(dA(\widetilde{e}))$ is the spectral radius of $dA(\widetilde{e})$.

Since $\omega(x)$ is not a singleton, we obtain that there exists a sequence $e_n\in \omega(x)\cap U$ such that $e_n\neq e$ and $e_n\to e$. Let $\widetilde{e}_n=\Phi(e_n)$. Then $\widetilde{e}-\widetilde{e}_n=A(\widetilde{e})-A(\widetilde{e}_n)=dA(\widetilde{e})(\widetilde{e}-\widetilde{e}_n)+o(\abs{\widetilde{e}-\widetilde{e}_n})$, where $o(\abs{ \widetilde{e}-\widetilde{e}_n}) / \abs{\widetilde{e}-\widetilde{e}_n} \to 0$ as $n\to \infty$. Let $w_n=\widetilde{e}-\widetilde{e}_n / \abs{\widetilde{e}-\widetilde{e}_n}$, then $w_n=dA(\widetilde{e})w_n+r_n$, where $r_n\to 0$ as $n\to \infty$. If $w_n \to w$ as $n\to \infty$, we obtain that $w=dA(\widetilde{e})w$. Hence, $\sigma(dA(\widetilde{e})) \geq 1$.

If $\sigma(dA(\widetilde{e})) =1$, then $w\in \text{Int} \widetilde{C}(\widetilde{e})$. Hence, there exists a neighborhood $\widetilde{N}$ of $w$ in the interior of $\widetilde{C}(\widetilde{e})$ such that there is a convex neighborhood $\widetilde{W}$ of $\widetilde{e}$ satisfies that for any $\widetilde{z}\in \widetilde{W}$, $\widetilde{N}\subset \widetilde{C}(\widetilde{z})$. Thus, there is a $N>0$ such that $\widetilde{e}_n\in \widetilde{W}$ and $w_n\in \widetilde{N}$ for all $n\geq N$. Let $\widetilde{\beta}(s)=s\widetilde{e}+(1-s)\widetilde{e}_n$ for $s\in [0,1]$, then $\frac{d}{ds}\widetilde{\beta}(s)=\widetilde{e}-\widetilde{e}_n\in \text{Int} \widetilde{C}(\widetilde{\beta}(s))$. Then, $\Phi^{-1}(\widetilde{\beta}(s))$ is a conal curve in $M$ connecting $e_n$ and $e$. Thus, $e_n\leq_M e$. Since $e_n,e\in \omega(x)$, it is a contradiction to Lemma \ref{fundamentalresults}(a). Thus, we obtain that $\rho(e) >1$. The conclusion of the lemma follows from the compactness of $\omega(x)$ and the continuity of the spectral radius (see e.g., \cite{Kato76,LemmensNussbaum}).
\end{proof}


\begin{prop}\label{improved limit set dichotomy}
If $x\leq_M y$, then either $\omega(x)\ll_M \omega(y)$, or $\omega(x)=\omega(y)=\{e\}$ for some $e\in E$.
\end{prop}
\begin{proof}
	We just need to prove the case that $\omega(x)=\omega(y)=\{e\}$ for some $e\in E$ by Lemma \ref{fundamentalresults}(b). Suppose that $\omega(x)=\omega(y)=K\subset E$. Let $\tau>0$ be the one in Lemma \ref{PerronFrobenius}. Let $x_n=\varphi_{n\tau}(x)$, $y_n=\varphi_{n\tau}(y)$ for $n\geq 1$. Since $x\leq_M y$, there exists a conal curve $\gamma : [0,1]\to M$ such that $\gamma(0)=x$, $\gamma(1)=y$ and $\frac{d}{ds}\gamma(s)\in C_M(\gamma(s))$. Let $\gamma_n(s)=\varphi_{n\tau}(\gamma(s))$, then $\gamma_n(s)$ is a strictly conal curve connecting $x_n$ and $y_n$ such that $x_n\ll_M y_n$. Suppose that $K$ contains more than a single element.

    By passing to a subsequence if necessary, we assume that $x_n \to p$, $y_n\to q$ as $n\to \infty$, where $p,q\in K$. Thus, $p=q$. Otherwise, $p\neq q$ with $p\leq_M q$, which is a contradiction to Lemma \ref{fundamentalresults}(a). The length of $\gamma_n$ is $L(\gamma_n)=\int_{0}^{1}\abs{\frac{d}{ds}\gamma_n(s)}_{\gamma_n(s)} ds$. We assert that $L(\gamma_n)\to 0$ as $n\to \infty$.

    $M$ is a globally orderable conal manifold, then every point in $M$ is a strong point (see \cite[Proposition 5.3]{Lawson89}). We claim that for any open neighborhood $U$ of $p$, there exists a $N>0$ such that $\gamma_n \subset U$ for all $n\geq N$. In fact, suppose there exist a neighborhood $U_1$ of $p$ and a subsequence $\{\gamma_{n_k}\}$ such that $\gamma_{n_k}$ leaves $U_1$ for all $k$. Since $p$ is a strong point, there exists $V_1\subset U_1$ open containing $p$ such that every conal curve that begins in $V_1$ and leaves $U_1$ terminates outsides of $V_1$ (see \cite[Lemma 5.2]{Lawson89}). On the other hand, $\gamma_{n_k}(0)\to p$ and $\gamma_{n_k}(1)\to p$, then there exists a $k_0>0$ such that for $k\geq k_0$, $\gamma_{n_k}(0)$ and $\gamma_{n_k}(1)$ are in $V_1$, which is a contradiction. Thus, we have the claim. Since $M$ is locally compact, we can find an open neighborhood $U$ of $p$ such that $\bar{U}$ is compact. By the previous claim, there is a $N$ such that $\gamma_{n}\subset U$ for all $n\geq N$. Let $\Omega=\{z\in M : z \text{ is a limit point of a sequence } \{z_n\}, z_n\in \gamma_{n}\}$. Clearly, $p\in \Omega$. If there is a $q\in \Omega$ with $q\neq p$, then there exist neighborhoods $W_1$ of $p$ and $W_2$ of $q$ such that $W_1\cap W_2=\emptyset$ since $M$ is a Hausdorff space. $q\in \Omega$ implies that there is a subsequence $\{z_{n_i}\}$ such that $z_{n_i}\to q$, where $z_{n_i}\in \gamma_{n_i}$. Then there is a $I_1>0$ such that for $i\geq I_1$, $z_{n_i}\in W_2$ with $z_{n_i}\notin W_1$, which contradicts the previous claim for neighborhood $W_1$ of $p$. Hence, $\Omega=p$. If there exist a $\alpha>0$ and a subsequence $n_j$ such that $L(\gamma_{n_j})\geq \alpha>0$, then $\Omega_1=\{z\in M : z \text{ is a limit point of a sequence } \{z_{n_j}\}, z_{n_j}\in \gamma_{n_j}\}=p$, which is a contradiction. So, $L(\gamma_n)\to 0$ as $n\to \infty$.

    Since $x_n$ is attracted to $K$, one can choose $e_n\in K$ such that $d(x_n,e_n)\to 0$ as $n\to \infty$. On the other hand, $L(\gamma_n)\to 0$ as $n\to \infty$. Hence, $\underset{0\leq s\leq 1}{\max} d(\gamma_n(s),e_n)\to 0$ as $n\to \infty$.
    Then there exist a $\bar{N}>0$ and a coordinate neighborhood $U$ of $p$ such that $\gamma_n, e_n\in U$ for all $n\geq \bar{N}$. Since there is a smooth coordinate chart $\Phi : U \to W\subset \mathbb{R}^n$, all notations in the following have coordinate representations. With the diffeomorphism $\Phi$, we treat the following notations both in manifold $M$ and $\mathbb{R}^n$.

    For each $s\in [0,1]$, we have
    {\small
    \begin{equation}\begin{split}\label{inequation}
    \abs{\frac{d}{ds}\gamma_{n+1}(s)}_{\gamma_{n+1}(s)} &=\abs{d\varphi_{\tau}(\gamma_n(s))\frac{d}{ds}\gamma_n(s)}_{\varphi_{\tau}(\gamma_n(s))} \\ & =\abs{\Gamma(\varphi_{\tau}(\gamma_n(s)),\varphi_{\tau}(e_n))d\varphi_{\tau}(\gamma_n(s))\frac{d}{ds}\gamma_n(s)}_{\varphi_{\tau}(e_n)} \\ & \geq - \abs{\Gamma(\varphi_{\tau}(\gamma_n(s)),\varphi_{\tau}(e_n))d\varphi_{\tau}(\gamma_n(s))\frac{d}{ds}\gamma_n(s)-d\varphi_{\tau}(e_n)\Gamma(\gamma_n(s),e_n)\frac{d}{ds}\gamma_n(s)}_{\varphi_{\tau}(e_n)}\\ & \quad +\abs{d\varphi_{\tau}(e_n)\Gamma(\gamma_n(s),e_n)\frac{d}{ds}\gamma_n(s)}_{\varphi_{\tau}(e_n)},
    \end{split}\end{equation}}
    where $\Gamma$ is defined as above.

    Since $d(\gamma_n(s),e_n)\to 0$ as $n \to \infty$,  and $\Gamma$ is continuous with $\Gamma(x,x)=\text{Id}_x$ for any $x\in M$, then for any $\delta_1>0$, there is a $N_1\geq \bar{N}$ such that for $n\geq N_1$, $\norm{\Gamma(\gamma_n(s),e_n)-\Gamma(e_n,e_n)}\leq \delta_1$. Since $\varphi_t$ is smooth, then for any $\delta_2>0$, there is a $N_2\geq \bar{N}$ such that for $n\geq N_2$, $\norm{\Gamma(\varphi_{\tau}(\gamma_n(s)),\varphi_{\tau}(e_n))d\varphi_{\tau}(\gamma_n(s))-\Gamma(\varphi_{\tau}(e_n),\varphi_{\tau}(e_n))d\varphi_{\tau}(e_n)} \leq \delta_2$. Let $\delta=\min \{\delta_1,\delta_2\}$ and $N=\max \{N_1,N_2\}$, then for $n\geq N$, we can obtain that $\norm{\Gamma(\gamma_n(s),e_n)-\Gamma(e_n,e_n)}\leq \delta$ and $\norm{\Gamma(\varphi_{\tau}(\gamma_n(s)),\varphi_{\tau}(e_n))d\varphi_{\tau}(\gamma_n(s))-\Gamma(\varphi_{\tau}(e_n),\varphi_{\tau}(e_n))d\varphi_{\tau}(e_n)} \leq \delta$. If there is a $L>0$ such that $\norm{d\varphi_{\tau}(x)}\leq L$ for all $x\in K$, then for $n\geq N$, we have that
    \begin{equation}\begin{split}\label{inequ1}
    &\abs{\Gamma(\varphi_{\tau}(\gamma_n(s)),\varphi_{\tau}(e_n))d\varphi_{\tau}(\gamma_n(s))\frac{d}{ds}\gamma_n(s)-d\varphi_{\tau}(e_n)\Gamma(\gamma_n(s),e_n)\frac{d}{ds}\gamma_n(s)}_{\varphi_{\tau}(e_n)} \\ \leq &\norm{\Gamma(\varphi_{\tau}(\gamma_n(s)),\varphi_{\tau}(e_n))d\varphi_{\tau}(\gamma_n(s))-d\varphi_{\tau}(e_n)\Gamma(\gamma_n(s),e_n)} \cdot\abs{\frac{d}{ds}\gamma_n(s)}_{\gamma_n(s)} \\
    \leq &\,(\norm{\Gamma(\varphi_{\tau}(\gamma_n(s)),\varphi_{\tau}(e_n))d\varphi_{\tau}(\gamma_n(s))-\Gamma(\varphi_{\tau}(e_n),\varphi_{\tau}(e_n))d\varphi_{\tau}(e_n)}\\& + \norm{d\varphi_{\tau}(e_n)\Gamma(\gamma_n(s),e_n)-d\varphi_{\tau}(e_n)\Gamma(e_n,e_n)})
     \cdot\abs{\frac{d}{ds}\gamma_n(s)}_{\gamma_n(s)}
    \\ \leq &\,(\delta+\delta \norm{d\varphi_{\tau}(e_n)})\cdot
    \abs{\frac{d}{ds}\gamma_n(s)}_{\gamma_n(s)}
    \\ \leq &\,\delta(1+L)\abs{\frac{d}{ds}\gamma_n(s)}_{\gamma_n(s)}.
    \end{split}\end{equation}

    Next, we deal with the term $\abs{d\varphi_{\tau}(e_n)\Gamma(\gamma_n(s),e_n)\frac{d}{ds}\gamma_n(s)}_{\varphi_{\tau}(e_n)}$ in inequation \eqref{inequation}. Due to Lemma \ref{PerronFrobenius}, we have that $d\varphi_{\tau}(e_n)v(e_n)=\rho(e_n) v(e_n)$, where $\rho(e_n)>\rho>1$ and $v$ is the Perron-Frobenius vector field. And we obtain that $\abs{\frac{1}{\abs{\frac{d}{ds}\gamma_n(s)}}\frac{d}{ds}\gamma_n(s)-v(\gamma_n(s))}_{\gamma_n(s)}\to 0$ as $n\to \infty$. Since $\Gamma$ preserves the metric, $\abs{\Gamma(\gamma_n(s),e_n)\frac{1}{\abs{\frac{d}{ds}\gamma_n(s)}}\frac{d}{ds}\gamma_n(s)-\Gamma(\gamma_n(s),e_n)v(\gamma_n(s))}_{e_n}\to 0$. Since $\Gamma(\cdot,\cdot)$ and $v(\cdot)$ are continuous, we have $\abs{\Gamma(\gamma_n(s),e_n)v(\gamma_n(s))-\Gamma(e_n,e_n)v(e_n)}_{e_n}\to 0$ as $d(\gamma_n(s),e_n)\to 0$. Hence, we obtain that $\abs{\Gamma(\gamma_n(s),e_n)\frac{1}{\abs{\frac{d}{ds}\gamma_n(s)}}\frac{d}{ds}\gamma_n(s)-v(e_n)}_{e_n}\to 0$ as $n\to \infty$. Since $\norm{ d\varphi_{\tau}(e_n)}$ is bounded for $e_n\in K$, let $\epsilon_n=\abs{d\varphi_{\tau}(e_n)\Gamma(\gamma_n(s),e_n)\frac{1}{\abs{\frac{d}{ds}\gamma_n(s)}}\frac{d}{ds}\gamma_n(s)-d\varphi_{\tau}(e_n)v(e_n)}_{\varphi_{\tau}(e_n)}$, then $\epsilon_n \to 0$ as $n\to \infty$ and
    we have that
    \begin{equation*}\begin{split}
    &\abs{d\varphi_{\tau}(e_n)\Gamma(\gamma_n(s),e_n)\frac{1}{\abs{\frac{d}{ds}\gamma_n(s)}}\frac{d}{ds}\gamma_n(s)}_{\varphi_{\tau}(e_n)} \\
    =&\,\abs{d\varphi_{\tau}(e_n)\Gamma(\gamma_n(s),e_n)\frac{1}{\abs{\frac{d}{ds}\gamma_n(s)}}\frac{d}{ds}\gamma_n(s) -d\varphi_{\tau}(e_n)v(e_n)+d\varphi_{\tau}(e_n)v(e_n)}_{\varphi_{\tau}(e_n)} \\
    \geq &\,\abs{d\varphi_{\tau}(e_n)v(e_n)}_{\varphi_{\tau}(e_n)}-\abs{d\varphi_{\tau}(e_n)\Gamma(\gamma_n(s),e_n)\frac{1}{\abs{ \frac{d}{ds}\gamma_n(s)}}\frac{d}{ds}\gamma_n(s)-d\varphi_{\tau}(e_n)v(e_n)}_{\varphi_{\tau}(e_n)} \\ =&\, \rho(e_n)\abs{v(e_n)}_{e_n}-\epsilon_n \\ >&\, \rho -\epsilon_n,
    \end{split}\end{equation*}
    Thus,
    \begin{equation}\begin{split}\label{inequ2}
    \abs{d\varphi_{\tau}(e_n)\Gamma(\gamma_n(s),e_n)\frac{d}{ds}\gamma_n(s)}_{\varphi_{\tau}(e_n)}
     > (\rho-\epsilon_n)\abs{\frac{d}{ds}\gamma_n(s)}_{\gamma_n(s)}.
    \end{split}\end{equation}

    It follows from \eqref{inequation}, \eqref{inequ1} and \eqref{inequ2} that
    {\small
    \begin{equation*}\begin{split}
    \abs{\frac{d}{ds}\gamma_{n+1}(s)}_{\gamma_{n+1}(s)}& \geq - \abs{\Gamma(\varphi_{\tau}(\gamma_n(s)),\varphi_{\tau}(e_n))d\varphi_{\tau}(\gamma_n(s))\frac{d}{ds}\gamma_n(s)-d\varphi_{\tau}(e_n)\Gamma(\gamma_n(s),e_n)\frac{d}{ds}\gamma_n(s)}_{\varphi_{\tau}(e_n)}\\ & \quad +\abs{d\varphi_{\tau}(e_n)\Gamma(\gamma_n(s),e_n)\frac{d}{ds}\gamma_n(s)}_{\varphi_{\tau}(e_n)} \\& \geq - \delta(1+L)\abs{ \frac{d}{ds}\gamma_n(s)}_{\gamma_n(s)} + (\rho-\epsilon_n)\abs{\frac{d}{ds}\gamma_n(s)}_{\gamma_n(s)} \\& = (\rho-\epsilon_n- \delta(1+L)) \abs{\frac{d}{ds}\gamma_n(s)}_{\gamma_n(s)}.
    \end{split}\end{equation*}
    }

    Since $\epsilon_n \to 0$ as $n\to \infty$ and $\rho>1$ by the Lemma \ref{PerronFrobenius}, we can choose $\delta$ small enough such that there exist a $\tilde{N}>N$ and $l>1$ satisfying for all $n\geq \tilde{N}$, $\rho-\epsilon_n- \delta(1+L)\geq l>1$. Thus, we obtain that $\abs{\frac{d}{ds}\gamma_{n+1}(s)}_{\gamma_{n+1}(s)} \geq l\abs{ \frac{d}{ds}\gamma_n(s)}_{\gamma_n(s)}$. So, $L(\gamma_{n+1})\geq lL(\gamma_{n})$, which is a contradiction to $L(\gamma_n)\to 0$ as $n\to \infty$.
    Thus, $K$ is a singleton.
\end{proof}

\begin{lem}\label{smooth section of cone field}
	Suppose that the cone field on $M$ admits a $C^{\infty}$-section $V$, then for each $x\in M$, there exist $\epsilon_x>0$ and a conal curve $\gamma_x : (-\epsilon_x,\epsilon_x) \to M$ such that $\gamma_x(0)=x$ and $\frac{d}{ds}\gamma_x(s)=V(\gamma_x(s))$ for $s\in (-\epsilon_x,\epsilon_x)$.
\end{lem}
\begin{proof}
	See \cite[Proposition 9.2]{LeeGTM218}. In fact, $\gamma_x$ is the integral curve of $V$.
\end{proof}
By the Lemma \ref{smooth section of cone field}, we obtain that for any $x\in M$, there exists a conal curve passing through $x$. Thus, we say that $x$ can be approximated from below (resp. above) in $M$ (i.e., there exists a sequence $\{x_n\}$ in $M$ such that $x_n\leq_M x_{n+1}\leq_M x$ (resp. $x\leq_M x_{n+1}\leq_M x_n$) for $n\geq 1$ and $x_n \to x$ as $n\to \infty$).

\begin{lem}\label{improved limit set trichotomy}
If $x_0\in M$ satisfies that there exists a sequence $\{x_n\}$ such that $x_n\leq_M x_{n+1}\leq_M x_0$ for $n\geq 1$ and $x_n\to x_0$, then one of the following alternatives must occur :
	
\begin{enumerate}[$(1)$]
	\item
	There exists $p\in E$ such that $\omega(x_n)\ll_M \omega(x_{n+1})\ll_M p=\omega(x_0)$ for all $n\geq 1$.
	
	\item
	$\omega(x_n)=\omega(x_0)=p\in E$ for $n\geq 1$.
		
	\item
	There exists $p\in E$ such that $\omega(x_n)=p\ll_M \omega(x_0)$ for all $n\geq 1$.
\end{enumerate}
\end{lem}

\begin{proof}
	Let $x_0\in M$ have the property that it is approximated from below in $M$ by a sequence $\{x_n\}$ such that $x_n\leq_M x_{n+1}\leq_M x_0$ for $n\geq 1$ and $x_n\to x_0$. By Proposition \ref{improved limit set dichotomy}, either there exists a $N>0$ such that $\omega(x_n)= \omega(x_m)$ for all $n,m\geq N$, or there is a subsequence $\{x_{n_i}\}$ such that $\omega(x_{n_i})\ll_M \omega(x_{n_{i+1}})$ for all $i\geq 1$. Thus, we assume that either $\omega(x_n)=\omega(x_m)$ for all $n,m$, or $\omega(x_n)\ll_M \omega(x_{n+1})$, where $x_n\to x_0$.
	
	Suppose that $\omega(x_n)\ll_M \omega(x_{n+1})$ for $n\geq 1$ and $x_n\to x_0$ as $n\to \infty$. We first obtain that $\omega(x_n)\ll_M \omega(x_0)$ for all $n\geq 1$. In fact, if there exists $n_0>0$ such that $\omega(x_{n_0})=\omega(x_0)$, then $\omega(x_n)=\omega(x_0)$ for all $n\geq n_0$, which is a contradiction. Let $\Omega=\{y : y=\lim\limits_{n \to \infty} y_n, y_n\in \omega(x_n)\}$. By the compactness condition $(P)$, the set $\Omega$ belogs to a compact set $\overline{\underset{n\geq 0}{\bigcup}\omega(x_n)}$. Furthermore, $\Omega$ is nonempty. If $\{y_{n_k}\}$ and $\{y_{m_k}\}$ are two subsequences of $\{y_n\}$ such that $y_{n_k}\to p$ and $y_{m_k}\to q$ as $k\to \infty$. Since $y_n\in \omega(x_n)$, then for each $k$, there exists $l(k)$ such that $y_{n_k}\ll_M y_{m_{l(k)}}$. Thus, $p\leq_M q$. A similar argument shows that $q\leq_M p$, then $p=q$. So, $\lim\limits_{n \to \infty}y_n=p$ and $p\in \Omega$. Thus, $\Omega$ is nonempty. Suppose that $u,w\in \Omega$ have the property that there exist two sequences $\{u_n\}$, $\{w_n\}$ with $u_n, w_n \in \omega(x_n)$ and $u_n\to u$, $w_n\to w$ as $n\to \infty$. Since $u_n\ll_M w_{n+1}$ (resp. $w_n\ll_M u_{n+1}$), we obtain that $u\leq_M w$ (resp. $w\leq_M u$). Then $u=w$. Thus, $\Omega$ is a singleton. On the other hand, since each $\omega(x_n)$ is invariant, we obtain that $\Omega$ is positively invariant. So, $\Omega=\{p\}\in E$. It is easy to see from the arbitrariness of $y_n$ in $\omega(x_n)$ that $\lim\limits_{n \to \infty}\text{dist}(\omega(x_n), p)=0$. Since $\omega(x_n)\ll_M \omega(x_0)$, then $p\leq_M \omega(x_0)$. If $p\in \omega(x_0)$, then $\omega(x_0)=p$ by Lemma \ref{fundamentalresults}(a). Thus, $x_0$ is a convergent point and $(1)$ holds. If $p\notin \omega(x_0)$, we will get a contradiction. Since $p\in E$ and $\omega(x_0)$ is invariant, we obtain that $p\ll_M \omega(x_0)$. For each $z\in \omega(x_0)$, there exist $t_z>0$, a neighborhood $U_z$ of $p$ and a neighborhood $V_z$ of $z$ such that $\varphi_t(U_z)\ll_M \varphi_t(V_z)$ for all $t\geq t_z$. Since $\{V_z\}$ is an open cover of $\omega(x_0)$, we obtain that $\omega(x_0)\subset \bigcup_{i=1}^{n}V_{z_i}=V$, where $z_i\in \omega(x_0)$. Meanwhile, $U=\bigcap_{i=1}^{n}U_{z_i}$ is a neighborhood of $p$. Let $t_0=\underset{1\leq i\leq n}{\max}\{t_{z_i}\}$, then $\varphi_t(U)\ll_M \varphi_t(V)$ for $t\geq t_0$. On the other hand, there exists $t_1>0$ such that $\varphi_{t_1}(x_0)\in V$. Since $x_n\to x_0$, there is a $N_1>0$ such that $\varphi_{t_1}(x_{N_1})\in V$. By $p\in E$, we have that $p\ll_M \varphi_t(x_{N_1})$ for $t>t_0+t_1$. Thus, $p\leq_M \omega(x_{N_1})$. By the definition of $\Omega$, $\omega(x_n)\ll_M \omega(x_{n+1})\leq_M p$ for all $n\geq 1$. Thus, $\omega(x_{N_1})=p$ and $\omega(x_n)=p$ for all $n>N_1$, which is a contradiction. Hence, we have proved the case $(1)$.
	
	Suppose that $\omega(x_n)=\omega(x_m)$ for all $n,m$. By Proposition \ref{improved limit set dichotomy}, $\omega(x_n)=p\in E$ for all $n>1$. Since $x_n\leq_M x_0$, then $\omega(x_n)=\omega(x_0)$ or $\omega(x_n)\ll_M \omega(x_0)$. So, $(2)$ and $(3)$ hold.
\end{proof}

\begin{rmk}\label{point is approximated from above}
	An analogous result holds if $x_0$ is approximated from above.
\end{rmk}

\vskip 4mm

Now, we are ready to prove Theorem \ref{convergent thm}.

\begin{proof}[Proof of Theorem \ref{convergent thm}]
	Suppose that $x_0\in M\backslash \text{Int} C$, then we will prove that $x_0\in \overline{\text{Int} C}$. In fact, there is a sequence $\{x_n\}\subset M\backslash C$ such that $x_n \to x_0$. Since the cone field admits a $C^{\infty}$-section, then all $x_n$ can be approximated from below and above in $M$. Without loss of generality, we assume that for each $x_n$, there exists a sequence $\{z_m^n\}$ such that $z_m^n\leq_M z_{m+1}^n\leq_M x_n$ for $m\geq 1$ and $z_m^n\to x_n$ as $m\to \infty$. Since $\{x_n\}\subset M\backslash C$ for all $n$, then the case $(3)$ of Lemma \ref{improved limit set trichotomy} holds for each sequence $\{z_m^n\}$.
	
	We claim that $x_n\in \overline{\text{Int} C}$ for each $n$. If the claim holds, then $x_0\in \overline{\text{Int} C}$ since $x_n \to x_0$ and $\overline{\text{Int} C}$ is a closed set. Thus, we obtain the theorem.
	
	Now, it suffices to prove the claim. For each $n$, $z_m^n\to x_n$ as $m\to \infty$ and $\{x_n\}\subset M\backslash C$, then there exists a $p\in E$ such that $\omega(z_m^n)=p\ll_M \omega(x_n)$ for all $m\geq 1$. For $y\in \omega(x_n)$, there exist a neighborhood $W_y$ of $y$ and $t_y>0$ such that $p\ll_M \varphi_t(W_y)$ for $t\geq t_y$. Since $\{W_y\}$ is an open cover of $\omega(x_n)$, we obtain that $\omega(x_n)\subset \bigcup_{i=1}^{k}W_{y_i}=W$, where $y_i\in \omega(x_n)$. Let $t_0=\underset{1\leq i\leq k}{\max}\{t_{y_i}\}$, then $p\ll_M \varphi_t(W)$ for $t\geq t_0$. On the other hand, there is a $t_1>0$ such that $\varphi_{t_1}(x_n)\in W$. Furthermore, one can choose a neighborhood $U$ of $x_n$ such that $\varphi_{t_1}(U)\subset W$. If $x\in U$ with $x\leq_M x_n$, then there exist a neighborhood $V$ of $x$, $V\subset U$, a neighborhood $O$ of $x_n$ and $t_2>0$ such that $\varphi_t(V)\ll_M \varphi_t(O)$ for $t\geq t_2$. We can choose a $L>0$ such that $z_L^n\in O$, then we have $\varphi_t(V)\ll_M \varphi_t(z_L^n)$ for $t\geq t_2$. Since $V\subset U$, $\varphi_{t_1}(U)\subset W$ and $p\ll_M \varphi_t(W)$ for $t\geq t_0$, we obtain that $p\ll_M \varphi_{t+t_1}(V)$ for $t\geq t_0$. Thus, $p\ll_M \varphi_t(V)\ll_M \varphi_t(z_L^n)$ for $t\geq t_0+t_1+t_2$. Since $\omega(z_m^n)=p$ for all $m\geq 1$, then $\omega(s)=p\in E$ for all $s\in V$. Hence, $x\in \text{Int} C$. Since the cone field admits a $C^{\infty}$-section, then there is a sequence $\{u_m^n\}$ in $U$ such that $u_m^n\leq_M x_n$ and $u_m^n\to x_n$ as $m\to \infty$. By the previous proof, $u_m\in \text{Int} C$. Thus, $x_n\in \overline{\text{Int} C}$. And hence, we have proved the claim.
\end{proof}

\section{Proof of Lemma \ref{fundamentalresults}}

In this section, we focus on proving the critical Lemma \ref{fundamentalresults}. We first need the following proposition.


\begin{prop}\label{convergence criterion}
	If $x\leq_M \varphi_T(x)$ for some $T>0$, then $\varphi_t(x) \to p \in E$ as $t \to \infty$.	
\end{prop}

\begin{proof}
	If $x\leq_M \varphi_T(x)$ and $\Sigma$ is SDP, then for $t_0>0$, there exist neighborhoods $U$ of $x$, $V$ of $\varphi_T(x)$ such that $\varphi_{t_0}(U)\ll_M \varphi_{t_0}(V)$ by Proposition \ref{strong flow open relation}.
	
	By the continuity, there exists $0<\epsilon_0<t_0$ such that $\varphi_{t_0}(x)\ll_M \varphi_{t_0+T+\epsilon}(x)$ for $\epsilon \in (-\epsilon_0, \epsilon_0)$.
	We first claim that $\omega(x)$ is a $\tau$-periodic orbit for any $\tau \in (T-\epsilon_0, T+\epsilon_0)$. We just prove the case of $\tau=T$ and a similar argument for $\tau \in (T-\epsilon_0, T+\epsilon_0)$. If $x\leq_M \varphi_T(x)$, then $\varphi_{nT}(x)\ll_M \varphi_{(n+1)T}(x)$ for $n=1,2,\cdots$. Thus, there exists $p\in M$ such that $\varphi_{nT}(x) \to p$ as $n \to \infty$. In fact, if there exist two sequences $\varphi_{n_{k}T}(x)$ and $\varphi_{n_{l}T}(x)$ satisfying $\varphi_{n_{k}T}(x) \to p$ and $\varphi_{n_{l}T}(x) \to q$, then for each $k$, there is a $l(k)$ such that $\varphi_{n_kT}(x)\ll_M \varphi_{n_{l(k)}T}(x)$. Thus, $p\leq_M q$ by (H2). A similar argument shows that $q\leq_M p$. So, $p=q$ by (H1). Consider the orbit of $p$, $\varphi_{t+T}(p)=\varphi_{t+T}(\lim\limits_{n\to \infty}\varphi_{nT}(x))=\lim\limits_{n \to \infty}\varphi_{(n+1)T+t}(x)=\varphi_t(p)$ for all $t\geq 0$. Thus, $O(p)$ is a $T$-periodic orbit. Suppose that there exist $q\in M$ and $\{t_j \}$ such that $t_j \to \infty$ and $\varphi_{t_j}(x)\to q$ as $j\to \infty$. For each $j$, $t_j=n_jT+s_j$, $n_j\in \{0, 1, 2, \cdots\}$, $0\leq s_j<T$, then $\varphi_{t_j}(x)=\varphi_{n_jT+s_j}(x)=\varphi_{s_j}(\varphi_{n_jT}(x))$. Thus, $\varphi_{t_j}(x)=\varphi_{s_j}(\varphi_{n_jT}(x))\to \varphi_s(p)$ as $j\to \infty$ where $s\in [0,T]$ such that $s_j \to s$ as $j \to \infty$. So, $\omega(x)= O(p)$.
	
	Since $\omega(x)$ is a $\tau$-periodic orbit for all $\tau\in (T-\epsilon_0, T+\epsilon_0)$ and $\omega(x)=O(p)$, then $\varphi_{t+\tau}(p)=\varphi_t(p)$ for all $t\geq0$ (i.e., $O(p)$ is $\tau$-periodic). Let $P$ be the set of all periods of $\varphi_t(p)$. It is easy to see that $(T-\epsilon_0, T+\epsilon_0)\subset P$ and $\varphi_{t+s}(p)=\varphi_t(\varphi_s(p))=\varphi_t(\varphi_{s+T}(p))=\varphi_t(p)$ for all $s\in [0,\epsilon_0)$ and $t\geq 0$. Thus, $[0,\epsilon_0)\subset P$. We next prove that $T+\epsilon_0\in P$. In fact, let $\epsilon\in (0,\epsilon_0)$ and $t=\epsilon_0-\epsilon \in (0,\epsilon_0)$, then $\varphi_{T+\epsilon_0}(p)=\varphi_{T+\epsilon+(\epsilon_0-\epsilon)}(p)=\varphi_{T+\epsilon}(\varphi_t(p))=\varphi_{T+\epsilon}(p)=p$. Thus, $T+\epsilon_0\in P$. Then $[0,\epsilon_0]\subset P$.
	
	For each $t>0$, $t=n\epsilon_0+\tau$, where $n\in \{0, 1, 2, \cdots\}$, and $0\leq \tau< \epsilon_0$. $\varphi_t(p)=\varphi_{n\epsilon_0+\tau}(p)=p$. Thus, $p\in E$ and $\omega(x)=p$.
\end{proof}

\vskip 3mm

Now, we prove Lemma \ref{fundamentalresults}{\rm (a)}.

\begin{proof}[Proof of Lemma \ref{fundamentalresults}{\rm (a)}]
If there exist two points $x,y\in \omega(z)$ with $x\leq_M y$, then we can find a neighborhood $U$ of $x$, a neighborhood $V$ of $y$ for $t_0>0$ such that $\varphi_{t_0}(U)\ll_M \varphi_{t_0}(V)$ by Proposition \ref{strong flow open relation}. Since $x,y\in \omega(z)$, there exist $0<t_1<t_2$ such that $\varphi_{t_1}(z)\in U$ and $\varphi_{t_2}(z)\in V$. Then $\varphi_{t_0+t_1}(z)\ll_M \varphi_{t_0+t_2}(z)= \varphi_{t_2-t_1}(\varphi_{t_0+t_1}(z))$. By Proposition \ref{convergence criterion}, $\varphi_t(z)\to p\in E$. So, $\omega(z)=p$, which is a contradiction.
\end{proof}

\vskip 3mm

In order to prove Lemma \ref{fundamentalresults}(b), we further need several propositions.

\begin{prop}\label{colimiting principle}
	If $x\leq_M y$, $t_k\to \infty$, $\varphi_{t_k}(x)\to p$ and $\varphi_{t_k}(y)\to p$ as $k\to \infty$, then $p\in E$.
\end{prop}
\begin{proof}
	Since $\Sigma$ is SDP, there exist a neighborhood $U$ of $x$, a neighborhood $V$ of $y$ for $t_0>0$ such that $\varphi_{t_0}(U)\ll_M \varphi_{t_0}(V)$ by Proposition \ref{strong flow open relation}. Let $\delta>0$ be small such that $\{\varphi_l(x) : 0\leq l\leq \delta \} \subset U$ and $\{\varphi_l(y) : 0\leq l\leq \delta \} \subset V$. Then $\varphi_s(x)=\varphi_{t_0}(\varphi_{s-t_0}(x))\ll_M \varphi_{t_0}(y)$ for any $t_0\leq s\leq t_0+\delta$. Thus, $\varphi_{t_k-t_0}(\varphi_s(x))\ll_M \varphi_{t_k-t_0}(\varphi_{t_0}(y))$ for $k$ large enough. Let $r=s-t_0$ and $k\to \infty$, we obtain $\varphi_r(p)\leq_M p$ for all $r\in [0, \delta]$. As a similar argument, we obtain $p\leq_M \varphi_r(p)$ for all $r\in [0, \delta]$. So, $p=\varphi_r(p)$ for all $r\in [0,\delta]$. For any $t>0$, we  write $t=n\delta+\tau$, where $\tau\in [0,\delta)$ and $n\in \{0, 1, 2, \cdots\}$. So, $\varphi_t(p)=\varphi_{n\delta+\tau}(p)=p$. Thus, $p\in E$.
\end{proof}

\begin{prop}\label{intersection principle}
	If $x\leq_M y$, then $\omega(x) \cap \omega(y)\subset E$.
\end{prop}
\begin{proof}
	Let $p\in \omega(x)\cap \omega(y)$, there exists a sequence $\{t_k\}$ such that $t_k\to \infty$ and $\varphi_{t_k}(x)\to p$ as $k\to \infty$. By a subsequence, we assume that $\varphi_{t_k}(y)\to q$ as $k\to \infty$. Since $x\leq_M y$ and $\Sigma$ is SDP, we obtain $p\leq_M q$ and $p,q\in \omega(y)$. The Lemma \ref{fundamentalresults}(a) means $p=q$ and Proposition \ref{colimiting principle} implies $p\in E$.
\end{proof}

\begin{lem}\label{separation lemma fixed point case}
	Let $x\leq_M y$, $p\in \omega(x)$, $q\in \omega(y)$ and $p\ll_M q$. If $p$ (or $q$) is an equilibrium, then $\omega(x)\ll_M \omega(y)$.	
\end{lem}
\begin{proof}
	Assume that $p\in E$. Since $q\in \omega(y)$ and $p\ll_M q$, there exists a sequence $\{t_k\}\to \infty$ such that $\varphi_{t_k}(y)\to q$ and hence, there exists a $k(q)>0$ such that $p\ll_M \varphi_{t_k(q)}(y)$. Thus, $p=\varphi_t(p)\ll_M \varphi_t(\varphi_{t_k(q)}(y))$ for all $t>0$. Thus, $p\leq_M \omega(y)$. By Lemma \ref{fundamentalresults}(a), $p\notin \omega(y)$. Moreover, $p\ll_M \omega(y)$. In fact, since $\omega(y)$ is invariant, for any $z\in \omega(y)$, there exists $T>0$ such that $\varphi_{-T}(z)\in \omega(y)$ and $p\leq_M \varphi_{-T}(z)$. Thus, $p\ll_M z$.
	
	Since $p\in \omega(x)$, there exists a sequence $\{t_l\}\to \infty$ such that $\varphi_{t_l}(x)\to p$.
	Since $p\ll_M \omega(y)$, there exists a $l(p)$ such that $\varphi_{t_l(p)}(x)\ll_M \omega(y)$. Thus, $\varphi_t(\varphi_{t_l(p)}(x))\ll_M \varphi_t(\omega(y))$ for all $t>0$. Since $\omega(y)$ is compact and invariant, we obtain that $\omega(x)\leq_M \omega(y)$. By Lemma \ref{fundamentalresults}(a), $\omega(x)\cap \omega(y)= \emptyset$. Since $\omega(x),\omega(y)$ are compact and invariant, then $\omega(x)\ll_M \omega(y)$. In fact, let $a\in \omega(x)$ and $b\in \omega(y)$, there is a $T_1>0$ such that $\varphi_{-T_1}(a)\in \omega(x)$ and $\varphi_{-T_1}(b)\in \omega(y)$ with $\varphi_{-T_1}(a)\leq_M \varphi_{-T_1}(b)$, then $a\ll_M b$.
	
	A similar argument is used if $q\in E$.
\end{proof}

\begin{lem}\label{perturbation lemma}
	Assume that $K\subset M$ is a compact set in which the flow $\varphi_t$ is SDP. Then there exists $\delta>0$ with the following property. Let $\psi_t$ denote the flow of a $C^1$ vector field $g$ such that $g\in U_{\delta}(f)$, where $U_{\delta}(f)$ is a $\delta$-neighborhood of $f$ in space of $C^1$ vector fields on $M$ with the $C^1$ topology. Then there exists $t_0>0$ such that if $K$ is positively invariant under the flow of $\psi_t$, then $\psi_t$ is SDP for all $t\geq t_0$.
\end{lem}	
\begin{proof}
	
	We first assume that $t\in [t_0, 2t_0]$, where $t_0>0$ is fixed. Since $\Sigma$ is SDP, we obtain that $d\varphi_t(x)v \in \text{Int} C_{M}(\varphi_t(x))$ for all $v\in C_M(x)$ with $\vert v \vert_x =1$. With a coordinate chart, we treat the following notations both in manifold and $\mathbb{R}^n$. By the continuity of the cone field, there exist neighborhoods $U$ of $\varphi_t(x)$ and $V$ of $d\varphi_t(x)v$ such that $V\subset C_M(y)$ for all $y\in U$. So, there exists $\delta>0$ such that for $g\in U_{\delta}(f)$, $\psi_t(x)\in U$ and $d\psi_t(x)v\in V$ (see e.g., \cite{Niu19,NiuXie23}). Thus, $d\psi_t(x)v\in \text{Int} C_M(\psi_t(x))$ for $t\in [t_0, 2t_0]$, $v\in C_M(x)\backslash \{0\}$.
	
	For $t\geq 2t_0$, let us write $t=r+kt_0$, where $r\in [t_0,2t_0)$ and $k\in \{1, 2, \cdots\}$. Define $x_j=\psi_{jt_0}(x)$, $j=1,2,\cdots,k$. It is clear that $x_j\in K$ if $K$ is positively invariant. Then $d\psi_t(x)v=d\psi_r(x_k)d\psi_{t_0}(x_{k-1})\cdots d\psi_{t_0}(x)v$. By the preceding proof, $d\psi_t(x)v\in \text{Int} C_M(\psi_t(x))$ for $t> 2t_0$.
	
	Thus, we have proved that $\psi_t$ is SDP for all $t\geq t_0$.
\end{proof}

\begin{lem}\label{separation lemma periodic case}
	Let $x\leq_M y$, $p\in \omega(x)$, $q\in \omega(y)$ and $p\ll_M q$. If $p$ (or $q$) belongs to a periodic orbit, then $\omega(x)\ll_M \omega(y)$.	
\end{lem}
\begin{proof}
	
	Assume that $p\in \gamma$, $\gamma$ is a periodic orbit and $q$ is not an equilibrium (the other case is similar).
	
	If $\gamma \cap \omega(y)\neq \emptyset$, then $p\subset \omega(y)$. Thus, $p\in E$ by Proposition \ref{intersection principle}, which is a contradiction. Thus, $\omega(y)\subset M \backslash \gamma$. Since $q\in \omega(y)$, then the orbit closure of $q$ is in the compact set $\omega(y)\subset M \backslash \gamma$.

	By the Closing Lemma (see e.g., \cite{H85,Pugh}), there is a $C^1$ vector field $g$ whose flow $\psi_t$ has a closed orbit $\beta_g$ passing through $q$. Moreover, $g$ can be chosen to $C^1$ approximate $f$ as closely as desired and to coincides with $f$ outside a given neighborhood $U$ of the orbit closure of $q$ with respect to $f$ such that $U\cap \gamma=\emptyset$. Thus, $\psi_t$ is eventually SDP by Lemma \ref{perturbation lemma} and $\gamma$ is also a closed orbit of $\psi_t$. In the following, we write $\gamma_f$ (resp. $\gamma_g$) as the orbit generated by $f$ (resp. $g$). So, $\gamma=\gamma_f=\gamma_g$.
	
	For the system generated by vector field $g$, $\gamma_g$ and $\beta_g$ are two periodic orbits and $p\in \gamma_g$, $q\in \beta_g$ with $p\ll_M q$. Then $\gamma_g\ll_M q$.
	
	In fact, if $p\ll_M q$, then there is a $\epsilon_0>0$ such that $\psi_\epsilon(p)\ll_M q$ for all $\epsilon\in [0, \epsilon_0]$. Let $T_1, T_2>0$ be the periods of $\gamma_g$ and $\beta_g$ and $a=\frac{T_2}{T_1}$. If $a$ is a rational number, let $a=\frac{n_1}{n_2}$, $n_1,n_2 \in \{1,2,\cdots\}$, then $n_1T_1=n_2T_2$. Define $\epsilon_1=\sup \{\epsilon > 0 : \psi_{\tau}(p)\ll_M q \text{ for all } \tau\in [0,\epsilon] \}$. Suppose that $\epsilon_1<\infty$. Then $\psi_{\epsilon_1}(p)\leq_M q$. Since $\psi_{\epsilon_1}(p)=\psi_{n_1T_1}(\psi_{\epsilon_1}(p))=\psi_{n_2T_2}(\psi_{\epsilon_1}(p))\ll_M \psi_{n_2T_2}(q)=q$, one has $\psi_{\epsilon_1}(p)\ll_M q$. Thus, there exists $\delta>0$ such that $\psi_{\epsilon_1+\epsilon}(p)\ll_M q$ for all $\epsilon\in [0,\delta]$, which is a contradiction. Thus, $\psi_{\epsilon}(p)\ll_M q$ for all $\epsilon \geq 0$. So, $\gamma_g\ll_M q$. If $a$ is an irrational number and $T_2=aT_1$, the set $W=\{\psi_{nT_2}(\psi_{\epsilon}(p)) : n=1,2,\cdots\}$ is dense in $\gamma_g$ for any fixed $\epsilon\in [0,\epsilon_0]$. For any $z\in \gamma_g$, there exists a sequence $\{z_i\}\subset W$ such that $z_i\ll_M q$ and $z_i\to z$ as $i\to \infty$.
	Thus, $\gamma_g\leq_M q$. So, $\gamma_g=\psi_{T_2}(\gamma_g)\ll_M q$.

	For system $\Sigma$, we obtain that $\gamma\ll_M q$. Since $q\in \omega(y)$, then there is a sequence $\{t_k\}\to \infty$ such that $\varphi_{t_k}(y)\to q$ as $k\to \infty$. Then there exists $k(q)>0$ such that $\gamma\ll_M \varphi_{t_{k(q)}}(y)$. Thus, we obtain that $\gamma\ll_M \varphi_t(\varphi_{t_k(q)}(y))$ for all $t>0$. So, $\gamma\leq_M \omega(y)$. On the other hand, $\gamma \cap \omega(y)=\emptyset$ and $\gamma,\omega(y)$ are invariant, then $\gamma\ll_M \omega(y)$. Since $\gamma\subset \omega(x)$, in a similar way we can obtain that $\omega(x)\ll_M \omega(y)$.
\end{proof}

\begin{prop}\label{separation lemma}
	Let $x\leq_M y$, $p\in \omega(x)$, $q\in \omega(y)$ and $p\ll_M q$. Then $\omega(x)\ll_M \omega(y)$.	
\end{prop}
\begin{proof}
	If $p\in E$, then $\omega(x)\ll_M \omega(y)$ by Lemma \ref{separation lemma fixed point case}. In the following, we assume that $p$ is not an equilibrium. Thus, $p\notin \omega(y)$.
	
	By the Closing Lemma (see e.g., \cite{H85,Pugh}), there exists $T>0$ such that for any $\epsilon>0$, we can choose a vector field $g$ such that $g=f$ outside the $\epsilon$-neighborhood $N_\epsilon$ of the set $\{\varphi_t(p) : \vert t \vert\leq T \}$. Moreover, $\psi_t$ with respect to the vector field $g$ has a closed orbit $\gamma_g$ passing through $p$ and $\psi_t$ is eventually SDP. For small $\epsilon$, $\omega(y)\subset M\backslash N_\epsilon$. It is easy to see that there exists a $y_0$ such that $p\leq_M y_0$ and $\omega_g(y_0)=\omega(y)$, where $\omega_g(y_0)=\{z\in M : \text{there exists a sequence }t_k \to \infty \text{ such that } \psi_{t_k}(y_0)\to z\}$.
	In fact, since $p\ll_M q$, there is a neighborhood $U_q$ of $q$ such that $U_q\cap N_\epsilon=\emptyset$ and $p\ll_M U_q$. Then there is a $t_q>0$ such that $\varphi_{t_q}(y)\in U_q$. Let $y_0=\varphi_{t_q}(y)$, then $p\ll_M y_0$ and $\omega_g(y_0)=\omega(y)$ since $g=f$ outside $N_\epsilon$ with $\omega(y)\subset M\backslash N_\epsilon$.
	Since $p\in \gamma_g$, $q\in \omega_g(y_0)$ and $p\ll_M q$, then $\gamma_g=\omega_g(p)\ll_M \omega_g(y_0)$ by Lemma \ref{separation lemma periodic case}. Thus, $p\ll_M \omega(y)$. Since $p\in \omega(x)$, there exists a sequence $\{t_k\}\to \infty$ such that $\varphi_{t_k}(x)\to p$. Then there exists $k(p)>0$ such that $\varphi_{t_{k(p)}}(x)\ll_M \omega(y)$. Thus, we obtain that $\varphi_t(\varphi_{t_{k(p)}}(x))\ll_M \varphi_t(\omega(y))$ for all $t>0$. So, $\omega(x)\leq_M \omega(y)$. Since $\omega(x)$ and $\omega(y)$ are nonordering invariant sets and $\Sigma$ is SDP, then $\omega(x)\cap \omega(y)= \emptyset$ and $\omega(x)\ll_M \omega(y)$.
\end{proof}

Now, we are ready to prove Lemma \ref{fundamentalresults}{\rm (b)}.

\begin{proof}[Proof of Lemma \ref{fundamentalresults}{\rm (b)}]
If $p\in \omega(x)$, then there exists a sequence $t_k \to \infty$ such that $\varphi_{t_k}(x)\to p$. By passing to a subsequence if necessary, we assume that $\varphi_{t_k}(y)\to q\in \omega(y)$. Since $\Sigma$ is SDP and $x\leq_M y$, we obtain that $p\leq_M q$. If $p\neq q$, then $\varphi_t(p)\ll_M \varphi_t(q)$ for any $t>0$. Thus, $\omega(x)\cap \omega(y)= \emptyset$ and $\omega(x)\ll_M \omega(y)$ by Proposition \ref{separation lemma}.

If $p=q$, then $p=q\in E$ by Proposition \ref{colimiting principle}. If $\omega(x)=\omega(y)$, then $\omega(x)=\omega(y)\subset E$ by Proposition \ref{intersection principle}. If $\omega(x)\neq \omega(y)$, then there exists $p_1\in \omega(x)\backslash \omega(y)$ such that there is a sequence $t_l \to \infty$ with $\varphi_{t_l}(x)\to p_1$. Therefore, there is $q_1\in \omega(y)$ such that $\varphi_{t_l}(y)\to q_1$ by a subsequence if necessary. Thus, $p_1\neq q_1$ and $p_1\leq_M q_1$. So, $\varphi_t(p_1)\ll_M \varphi_t(q_1)$ for any $t>0$ and hence, we obtain that $\omega(x)\cap \omega(y)= \emptyset$ and $\omega(x)\ll_M \omega(y)$ by Proposition \ref{separation lemma}, which is a contradiction.
\end{proof}

\appendix

\section{Appendix}

\subsection{Order structures on space-times}

A \textit{Lorentz metric} $g$ for a smooth manifold $M$ of dimension four is a smooth nondegenerate symmetric tensor field of type $(0,2)$ on $M$ such that for each $p\in M$, by suitable choice of the basis, $g\vert_{p}$ has the matrix diag$(+1, +1, +1, -1)$.	
A \textit{space-time} $(M,g)$ is a connected $C^{\infty}$ Hausdorff manifold $M$ of dimension four with a Lorentz metric $g$.

\begin{rmk}
	Although the arguments refer to $4$-dimensional space-times, the results can be extended to a space-time of $n(\geq 2)$-dimensions. See \cite{Beemandehrlich81,Hawkingandellis73,Penrose72}.
\end{rmk}

Let $(M,g)$ be a space-time, a vector $v\in T_pM$ is said to be \textit{timelike}, \textit{null}, \textit{spacelike} according to whether $g(v,v)$ is negative, zero, or positive, respectively.
The non-spacelike (i.e., timelike and null) vectors in $T_pM$ form two so-called \textit{Lorentzian cones} $C$ and $-C$ (see e.g., \cite{Hawkingandellis73}). Furthermore, the timelike vectors form the interior of the Lorentzian cones. See Fig.\ref{nullcone}.

\begin{figure}[th]
	\centering
	\includegraphics[width=0.7\columnwidth]{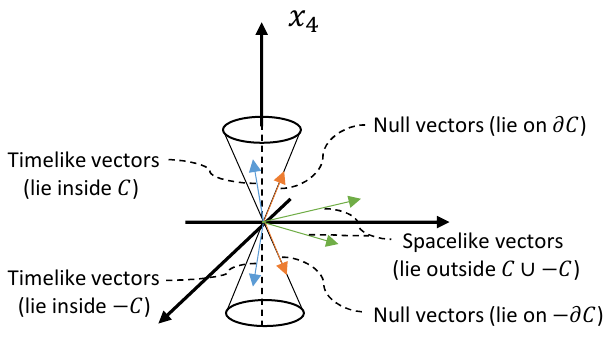}\label{nullcone}
	\center{{\bf Fig A.1}: The timelike vectors, null vectors and spacelike vectors in $T_pM$.}
\end{figure}

A space-time $(M,g)$ is said to be \textit{time-orientable} if $M$ admits a continuous Lorentzian cone field $C_M$, which is generated by Lorentz metric $g$.

In the remained of this subsection, we assume that space-time $(M,g)$ is time-orientable. Thus, $M$ is a conal manifold with respect to a continuous Lorentzian cone field $C_M$. In such situation, the non-spacelike vectors, which belong to cone field $C_M$, are called future directed.

A \textit{non-spacelike curve} is a continuous piecewise smooth curve whose tangent vector is future directed non-spacelike.
A \textit{timelike curve} is a continuous piecewise smooth curve whose tangent vector is future directed timelike.
Thus, in a time-orientable space-time, a non-spacelike (resp., timelike) curve is a conal (resp., strictly conal) curve and vice versa.
The order ``$\leq_M$" (resp., ``$\ll_M$") on $M$ is well-defined by the non-spacelike (resp., timelike) curves.

A non-spacelike curve is also known as \textit{causal} curve. In general relativity, each point of manifold $M$ corresponds to an event. And a signal can be sent from $p$ to $q$ if there is a (future directed) causal curve from $p$ to $q$. Thus, closed causal curves generate paradoxes involving causality (i.e., violate causality). As a result, we assume that the space-time $(M,g)$ is \textit{causal}, i.e., $(M,g)$ contain no closed non-spacelike (conal) curves (see e.g., \cite{Beemandehrlich81,Hawkingandellis73,Hawkingandsachs74,Penrose72}).

For a given point $p\in M$, the \textit{chronological future} $I^+(p)$, \textit{chronological past} $I^-(p)$, \textit{causal future} $J^+(p)$, and \textit{causal past} $J^-(p)$ of $p$ are defined as follows:
\begin{itemize}
	\item  $I^+(p)=\{q\in M : p\ll_M q\}$; \ \  $I^-(p)=\{q\in M : q\ll_M p\}$;
	\item  $J^+(p)=\{q\in M : p\leq_M q\}$; \ \  $J^-(p)=\{q\in M : q\leq_M p\}$.
\end{itemize}

\begin{rmk}
	In some articles, the set $J^+(p)$ would be called \textit{forward set} or \textit{reachable set} from $p$ and be written as $\uparrow p$; the set $J^-(p)$ would be called \textit{backward set} or \textit{controllable set} from $p$ and be written as $\downarrow p$.
\end{rmk}

\begin{rmk}
	For any $p\in M$, $I^{\pm}(p)$ are open by Proposition \ref{strong order open property}.
\end{rmk}

$I^+$ is said to be \textit{inner continuous} at $p\in M$ if each compact set $K\subset I^+(p)$, there exists a neighborhood $U(p)$ of $p$ such that $K\subset I^+(q)$ for each $q\in U(p)$.
$I^+$ is said to be \textit{outer continuous} at $p\in M$ if each compact set $K\subset M\backslash \overline{I^+(p)}$, there exists a neighborhood $U(p)$ of $p$ such that $K\subset M\backslash \overline{I^+(q)}$ for each $q\in U(p)$. The \textit{inner} and \textit{outer continuity} of $I^-$ are defined dually. See \cite{Hawkingandsachs74}.

\begin{prop}\label{innercontinuous}
	For any $p\in M$, $I^{\pm}(p)$ are inner continuous.
\end{prop}	

\begin{proof}
	Suppose $K\subset I^{-}(p)$ is compact. For any $x\in I^-(p)$, i.e., $x\ll_M p$, there is a strictly conal curve $\gamma$ such that $\gamma (0)=x$ and $\gamma(1)=p$. Let $w=\gamma(\frac{1}{2})$, then $x\ll_M w\ll_M p$, i.e., $w\in I^-(p)$ and $x\in I^-(w)$. Since $I^-(w)$ is open, then $I^-(w)$ is an open neighborhood of $x$. Thus, $\{I^-(w): w\in I^-(p)\}$ is an open covering of $K$. Since $K$ is a compact set, we choose $w_1, w_2, \cdots, w_n$ to determine a finite subcovering. On the other hand, $w_i\in I^-(p)$ implies that $p\in I^+(w_i)$, $i=1,2,\cdots, n$. So, $U=\bigcap^{n}_{i=1}I^+(w_i)$ is an open neighborhood of $p$. For any $u\in U$, $u\in I^+(w_i)$, $i=1,2,\cdots, n$. Then $w_i\in I^-(u)$. For any $y\in I^-(w_i)$, since $y\ll_M w_i$ and $w_i\ll_M u$, then $y\ll_M u$, i.e., $y\in I^-(u)$. So, $I^-(w_i)\subset I^-(u)$ for $i=1,2,\cdots,n$. Thus, $K\subset \bigcup ^{n}_{i=1}I^-(w_i)\subset I^-(u)$. So, we have proved that $I^{-}(p)$ is inner continuous.
	
	A similar argument is used for $I^+(p)$.
\end{proof}

\begin{prop}\label{transitive}
	\begin{enumerate}[$(1)$]
		\item $x\ll_M y$, $y\leq_M z$ implies $x\ll_M z$,
		\item $x\leq_M y$, $y\ll_M z$ implies $x\ll_M z$.
	\end{enumerate}
\end{prop}

\begin{proof}
	See \cite[p183]{Hawkingandellis73} and \cite[Proposition 2.18]{Penrose72}.
\end{proof}

\begin{rmk}
	In \cite{ChruscielandGrant12,Garcia2021}, Proposition \ref{transitive} is called \textit{push-up lemma}. The push-up lemma fails in the case where the spacetime metric (i.e., Lorentz metric) is continuous, see \cite[Example 1.11]{ChruscielandGrant12} or \cite{Garcia2021}.
	
\end{rmk}

\begin{prop}\label{closedprop}
	$I^{\pm}(p)\subset J^{\pm}(p)\subset \overline{I^{\pm}(p)}=\{q\in M: I^{\pm}(q)\subset I^{\pm}(p)\}$.
\end{prop}

\begin{proof}
	It is clear that $I^+(p)\subset J^+(p)$. See \cite[Proposition 3.3]{Penrose72} for $\overline{I^+(p)}=\{q\in M: I^+(q)\subset I^+(p)\}$ and \cite[Proposition 3.9]{Penrose72} for $J^+(p)\subset \overline{I^+(p)}$.
\end{proof}

\begin{thm}\label{quasiclosed}
	For any $p\in M$, if $J^{\pm}(p)$ are closed and $I^{\pm}(p)$ are outer continuous, then the conal order ``$\leq_M$" is quasi-closed.
\end{thm}

\begin{proof}
	If $x_n\ll_M y_n$ for all $n$ and $x_n\to x$ and $y_n\to y$ as $n\to \infty$, we just need to prove that $y\in J^+(x)= \overline{I^+(x)}$ by Proposition \ref{closedprop}. Suppose that $y\notin \overline{I^+(x)}$, i.e., $y\in  M\backslash \overline{I^+(x)}$. Then there exists a compact set $K$ containing $y$ such that $K\subset M\backslash \overline{I^+(x)}$. Since $y_n\to y$ as $n\to \infty$, there is a $N_1>0$ such that $y_n\in K$ for all $n>N_1$. On the other hand, $I^+$ is outer continuous at $x$, then there exists a neighborhood $U$ of $x$ such that for any $z\in U$, $K\subset M\backslash \overline{I^+(z)}$. Since $x_n\to x$ as $n\to \infty$, there is a $N_2>0$ such that $x_n\in U$ and $K\subset M\backslash \overline{I^+(x_n)}$ for all $n>N_2$. Let $N=\max \{N_1,N_2\}$. If $n>N$, then $x_n\in U$ and $y_n\in K$ with $y_n\in I^+(x_n)$, which is a contradiction. Thus, $y\in J^+(x)= \overline{I^+(x)}$.
\end{proof}

\begin{prop}\label{reflecting}
	The following conditions are equivalent.
	\begin{enumerate}[{\rm (A)}]
		\item
		For all $p$ and $q$ in $M$, $I^+(p)\subset I^+(q)$ if and only if $I^-(q)\subset I^-(p)$;
		
		\item
		For all $p$ and $q$ in $M$, $p\in \overline{J^+(q)}$ if and only if $q\in \overline{J^-(p)}$.
	\end{enumerate}
\end{prop}

\begin{proof}
	The conclusion can be immediately obtained with the Proposition \ref{closedprop}. See also \cite[Proposition 1.3]{Hawkingandsachs74}.
\end{proof}

\begin{prop}\label{outercontinuous}
	If any one of the equivalent conditions in Proposition \ref{reflecting} holds, then $I^{\pm}(p)$ are outer continuous for any $p\in M$.
\end{prop}

\begin{proof}
	We first assert that for $v\in M\backslash \overline{I^-(p)}$, there exists $w\in I^+(p)$ such that $v\in M\backslash \overline{I^-(w)}$.
	Before proving this assertion, we will show how it implies this Proposition. Let $K\subset M\backslash \overline{I^-(p)}$ be compact. This assertion implies that $\{M\backslash \overline{I^-(w)}: w\in I^+(p)\}$ is an open covering of $K\subset M\backslash \overline{I^-(p)}$. Choose a finite subcovering $\{M\backslash \overline{I^-(w_i)}:w_i\in I^+(p), i=1,2,\cdots,n\}$ determined by $w_1,w_2,\cdots,w_n$ and $U=\bigcap^{n}_{i=1}I^-(w_i)$. Then $U$ is a neighborhood of $p$ such that for any $u\in U$, $K\subset M\backslash \overline{I^-(u)}$. In fact, since $w_i\in I^+(p)$, then $p\in I^-(w_i)$, $i=1,2,\cdots,n$. So, $p\in \bigcap^{n}_{i=1}I^-(w_i)=U$. For any  $u\in U$, $u\in I^-(w_i)$ for $i=1,2,\cdots,n$. Then $I^-(u)\subset I^-(w_i)$ for $i=1,2,\cdots,n$. Thus, $M\backslash \overline{I^-(w_i)} \subset M\backslash \overline{I^-(u)}$ for $i=1,2,\cdots,n$. So, $K\subset M\backslash \overline{I^-(u)}$. Then $I^-(p)$ is outer continuous.
	
	It remains to prove the assertion. Suppose that $v\in \overline{I^-(w)}$ for any $w\in I^+(p)$, then we will get a contradiction. Since $v\in \overline{I^-(w)}$, then $I^-(v)\subset I^-(w)$ by Proposition \ref{closedprop}. Since condition $(A)$ of Proposition \ref{reflecting} holds, then $I^+(w)\subset I^+(v)$ for all $w\in I^+(p)$. On the other hand, for any $y\in I^+(p)$, i.e., $p\ll_M y$, there is a strictly conal curve $\gamma : [0,1] \mapsto M$ such that $\gamma(0)=p$ and $\gamma(1)=y$. Let $x=\gamma(\frac{1}{2})$, then $p\ll_M x\ll_M y$. Thus, $y\in I^+(x)$ and $x\in I^+(p)$. Since $I^+(w)\subset I^+(v)$ for all $w\in I^+(p)$, then $I^+(x)\subset I^+(v)$. So, $y\in I^+(v)$ for any $y\in I^+(p)$. Then $I^+(p)\subset I^+(v)$. By Proposition \ref{closedprop}, $p\in \overline{I^+(v)}$. Thus, $v\in \overline{I^-(p)}$ by Proposition \ref{reflecting}, which is a contradiction with $v\in M\backslash \overline{I^-(p)}$.
\end{proof}

\begin{rmk}
	If for all $p$ and $q$ in $M$, either $I^+(p)=I^+(q)$ or $I^-(p)=I^-(q)$ implies $p=q$, then that $I^{\pm}$ are outer continuous is equivalent to any one of the conditions in Proposition \ref{reflecting}. See \cite[Theorem 2.1]{Hawkingandsachs74}.
\end{rmk}

\begin{lem}\label{causallysimple}
	If $J^{\pm}(p)$ are closed for all $p\in M$, then $I^{\pm}(p)$ are outer continuous.
\end{lem}

\begin{proof}
	If $J^{\pm}(p)$ are closed for all $p\in M$, then condition (B) of Proposition \ref{reflecting} holds. So, $I^{\pm}(p)$ are outer continuous by Proposition \ref{outercontinuous}.
\end{proof}

\begin{cor}\label{quasiclosedcor}
	If $J^{\pm}(p)$ are closed for all $p\in M$, then the conal order ``$\leq_M$" is quasi-closed.
\end{cor}

\begin{proof}
	If $J^{\pm}(p)$ are closed for all $p\in M$, then $I^{\pm}(p)$ are outer continuous by Lemma \ref{causallysimple}. Thus, the conal order ``$\leq_M$" is quasi-closed by Theorem \ref{quasiclosed}.
\end{proof}



\subsection{Globally orderable homogeneous spaces}

Let $G$ be a connected Lie group and $M$ be a smooth manifold. A \textit{left action} of Lie group $G$ on manifold $M$ is a smooth map $\theta : G \times M \to M$ satisfying $\theta(e,x)=x$ and $\theta(g_1g_2,x)=\theta(g_1,\theta(g_2,x))$ for all $g_1,g_2\in G$, $x\in M$, where $e$ is the identity element in $G$. We write $g\cdot x$ or $gx$ for $\theta(g,x)$. The action is said to be \textit{transitive} if for every pair of points $x,y\in M$, there exists $g\in G$ such that $g\cdot x=y$. For each $x\in M$, the \textit{isotropy group} of $x$, denoted by $G_x$, is the set $G_x=\{g\in G : g\cdot x=x\}$.

A smooth manifold endowed with a transitive smooth action by a Lie group $G$ is called a \textit{homogeneous G-space} (or a \textit{homogeneous space}).

Let $G$ be a Lie group and $H\subset G$ be a closed subgroup. The \textit{left coset space} $G/H=\{gH : g\in G\}$ is a smooth manifold of dimension ($\text{dim} G-\text{dim}H$), and the left action of $G$ on $G/H$ is given by $g_1\cdot (g_2H)=(g_1g_2)H$. Hence, $G/H$ is a homogeneous space (see \cite[Theorem 21.17]{LeeGTM218}).

Let $G$ be a Lie group and $M$ be a homogeneous G-space. If $p$ is any point of $M$, then the isotropy group $G_p$ is a cloed subgroup of $G$, and the $F : G/G_p \to M$ defined by $F(gG_p)=g\cdot p$ is an equivariant diffeomorphism (see \cite[Theorem 21.18]{LeeGTM218}). Because of this equivariant diffeomorphism, we can define a homogeneous space to be a coset space of the form $G/H$, where $G$ is a Lie group and $H$ is a closed subgroup of $G$.

Let $G$ be a Lie group and fix $a\in G$. Define the \textit{left translation} map $L_a : G \to G$ by $L_a(g)=ag$. The left translation map is diffeomorphism since it is smooth with smooth inverse. The inverse of $L_a$ is clearly the map $L_{a^{-1}}$. The diffeomorphism $L_g$ induces a vector space isomorphism $dL_g\vert_e : \mathfrak{g}=T_eG \to T_gG$, where $\mathfrak{g}$ is the Lie algebra of $G$ and $e$ is the identity element in $G$.

Let $M=G/H$ be a homogeneous space and the natural projection $\pi : G\to G/H$, $\pi(g)=gH$ be a submersion. For each $a\in G$, define the \textit{left translation} $\tau_a : G/H \to G/H$ by $\tau_a(gH)=agH$. Then the left translations $\tau_g$ are related to the left translations $L_g$ on the Lie group $G$ by $\pi \circ L_g = \tau_g \circ \pi$ for each $g\in G$. Let $\mathfrak{h}$ be the Lie algebra of $H$ and $o=\pi(e)=eH$. The differential $d\pi\vert_e : T_eG\to T_oM$ is a vector space homomorphism with $\ker d\pi\vert_e=\mathfrak{h}$, we obtain that $\mathfrak{g}/\mathfrak{h}\cong T_oM$, where $\mathfrak{g}/\mathfrak{h}$ is the set of cosets $X+\mathfrak{h}=\{X+Y:Y\in \mathfrak{h}\}$ for $X\in \mathfrak{g}$ (see \cite[P488]{HilgertHofmannLawson} or \cite[P71]{Arvanitoyeorgos}).

A \textit{wedge} $W$ is a closed and convex subset of a vector space that is invariant by scaling with real positive numbers (see e.g., \cite{HilgertHofmannLawson} and \cite{KLS89}). Thus, a convex cone is a wedge in a vector space. A \textit{wedge field} $W_M$ on a manifold $M$ assigns to each point $x\in M$ a wedge $W_M(x)$ in the tangent space $T_xM$.

Let $\Phi : G\times M \to M$ be any left group action on $M$ such that each of the maps $\tau_g : M\to M$ defined by $\tau_g(x)=\Phi(g,x)=g\cdot x$ forms a diffeomorphism of $M$, Then a wedge field $W_M$ is said to be \textit{G-invariant} or \textit{homogeneous} if $d\tau_g \vert_x (W_M(x))=W_M(g\cdot x)$ for all $g\in G$ and $x\in M$.

\begin{lem}\label{invariantconefieldfromwedge}
Let $H$ be a closed subgroup of a Lie group $G$ and $W$ a wedge in $\mathfrak{g}$ such that (i) $W\cap -W=\mathfrak{h}$, and (ii) $\mathrm{Ad}(H)(W)=W$, where $\mathrm{Ad}$ is the adjoint representation of $G$. Define $W_G$ and $C_M$ by
\begin{equation*}
W_G(g)=dL_g\vert_e W, \ \ C_M(x)=d\tau_g\vert_o C,
\end{equation*}
where $M=G/H$, $e$ is the identity element in $G$, $o=eH$ is the base point in $M$, $x=gH$, and $C$ is the convex cone in $T_oM$ obtained as the projection of $W$ onto $\mathfrak{g}/\mathfrak{h}$. Then, $W_G$ is an invariant wedge field on $G$ and $C_M$ is a well-defined homogeneous or G-invariant cone field on $M$. Moreover, for each $g\in G$,
\begin{equation*}
d\pi \vert_g (W_G)= C_M(\pi(g)),
\end{equation*}
where $\pi : G\to M$ is the canonical projection $\pi (g)=gH$.
\end{lem}
\begin{proof}
See \cite[Lemma VI.1.5]{HilgertHofmannLawson}.
\end{proof}

\begin{prop}\label{homogeneousconecontinuous}
The homogeneous cone field on a homogeneous space is continuous and admits $C^{\infty}$ sections.
\end{prop}
\begin{proof}
See \cite[Proposition 4.6]{Lawson89}.
\end{proof}

\begin{thm}\label{globalorderforhomogeneouscone}
	Let $C_M$ be a homogeneous cone field on $M=G/H$ as described in Lemma \ref{invariantconefieldfromwedge}. If $S=\overline{<\exp W>H}\subset G$, then $S=\pi^{-1}(\overline{\{x\in M : o\leq_M x\}})$ and $G/H$ is globally orderable with respect to $C_M$ if and only if $W=\textbf{L}(S)$, where $\textbf{L}(S)=\{Z\in \mathfrak{g}: \exp (\mathbb{R}^{+}Z) \subset S\}$.
\end{thm}
\begin{proof}
This Theorem is derived from \cite[Theorem 1.6]{Neeb91} or \cite[Theorem 4.21]{HilgertNeeb93}. Where $\exp : \mathfrak{g} \to G$ is the Lie group exponential map, $\exp (tZ)$ with $Z\in \mathfrak{g}, t\in \mathbb{R}$ is the one-parameter subgroup on $G$ (see e.g., \cite{LeeGTM218}) and $\leq_M$ is the order generated by cone field $C_M$.
\end{proof}

Let $M=G/H$ be a homogeneous space with base-point $o=eH$. A Riemannian metric $(\cdot,\cdot)_p$, $p\in M$, is said to be \textit{G-invariant} or \textit{homogeneous}, if it satisfies $(X,Y)_o=(d\tau_g(X),d\tau_g(Y))_{g\cdot o}$ for each $g\in G$ and $X,Y\in T_oM$.

\begin{prop}\label{homogeneouscomplete}
Let $G$ be a Lie group, $H$ a closed subgroup, then the space $G/H$ is complete in any G-invariant metric.
\end{prop}
\begin{proof}
	See \cite[p148]{Helgason}.
\end{proof}

A homogeneous space $M=G/H$ is called \textit{reductive} if there exists a subspace $\mathfrak{m}$ of $\mathfrak{g}$ such that $\mathfrak{g}=\mathfrak{h}\oplus \mathfrak{m}$ and $\text{Ad}(h)\mathfrak{m}\subset \mathfrak{m}$ for all $h\in H$. Hence, $\mathfrak{m}\cong T_oM$ (see \cite{Arvanitoyeorgos}). The next Proposition gives a simple description of G-invariant Riemannian metrics on a homogeneous space.

\begin{prop}\label{homogeneousmetric}
Let $M=G/H$ be a reductive homogeneous space. Then there is a one-to-one correspondence between G-invariant Riemannian metrics on $M=G/H$ and $\mathrm{Ad}^{G/H}$-invariant inner products $\langle \cdot,\cdot \rangle$ on $\mathfrak{m}\cong T_oM$; that is, $\langle X,Y \rangle=\langle \mathrm{Ad}^{G/H}(h)X,\mathrm{Ad}^{G/H}(h)Y \rangle$ for all $X,Y\in \mathfrak{m}$, $h\in H$.
\end{prop}
\begin{proof}
See \cite[Proposition 5.1]{Arvanitoyeorgos}. The homomorphism $\mathrm{Ad}^{G/H} : H \to GL(T_oM)$ such that $\mathrm{Ad}^{G/H}(h)=d\tau_h\vert_o$ is the isotropy representation of the homogeneous space $M=G/H$.
\end{proof}

\subsection{Differential positivity in flat spaces}
Let $M$ be the $n$-dimensional Eucliean space $\mathbb{R}^n$ and $C$ is a closed convex cone of $M$.
There is a partial order ``$\leq$" on $M$ generated by cone $C$ ($x\leq y$ if and only if $y-x\in C$).
It should be pointed out that the order introduced by a closed convex cone is a closed partial order.

We consider the cone field on $M$ defined by $C_M(x)=\{x\}\times C\subset TM=\mathbb{R}^n\times \mathbb{R}^n$. Such cone field is said to be \textit{a contant cone field}. Thus, we can define the order ``$\leq_M$" on $M$ with respect to cone field $C_M$ ($x\leq_M y$ if and only if there exists a conal curve $\gamma : [0,1]\to M$ such that $\gamma(0)=x$, $\gamma(1)=y$ and $\frac{d}{ds}\gamma(s)\in C_M(\gamma(s))=C$).

It is easy to see that the contant cone field on Eucliean space $\mathbb{R}^n$ satisfies the smoothness conditions and invariance condition in the previous sections.

The next proposition implies that the order ``$\leq_M$" generated by contant cone feild $C_M$ agrees with the partial order ``$\leq$" on $M$ generated by cone $C$ (see \cite[Proposition 1.10]{Lawson89}).

\begin{prop}\label{order equivalence}
	Let $M$ be a $n$-dimensional Eucliean space and $C$ is a convex cone in $M$ such that $C$ forms a constant cone field $C_M$. Then for $x,y\in M$, $x\leq y$ if and only if $x\leq_M y$.
\end{prop}
\begin{proof}
	If $x\leq y$, i.e., $y-x\in C$, then we choose a curve $\alpha(s)=sy+(1-s)x$, where $s\in [0,1]$. Thus, $\alpha(0)=x$, $\alpha(1)=y$, and $\frac{d}{ds}\alpha(s)=y-x\in C$. So, $\alpha(s)$ is a conal curve and $x\leq_M y$.
	
	If $x\leq_M y$, then there is a conal curve $\alpha : [0,1] \to M$ such that $\alpha(0)=x$, $\alpha(1)=y$ and $\frac{d}{ds}\alpha(s)\in C$. For any $\lambda\in C^*\backslash \{0\}$, where $C^*$ is the dual cone of $C$, $\lambda(\frac{d}{ds}\alpha(s))\geq 0$. On the other hand, $\lambda(y-x)=\lambda(\alpha(1)-\alpha(0))=\lambda(\alpha(1))-\lambda(\alpha(0))$. Since $\frac{d}{ds}\lambda(\alpha(s))=\lambda(\frac{d}{ds}\alpha(s))\geq 0$, then $\lambda(\alpha(1))\geq \lambda(\alpha(0))$. So, $\lambda(y-x)\geq 0$. Thus, we obtain that $y-x\in C$, i.e., $x\leq y$.
\end{proof}
And a similar result can be obtained for the order ``$\ll$" with ``$\ll_M$", where $x\ll y$ if and only if $y-x\in \text{Int} C$.

Let $\frac{dx}{ds}=f(x)$ be a dynamical system in $\mathbb{R}^n$ with the flow $\varphi_t$. The system is said to be \textit{monotone} with respect to partial order ``$\leq$" if $\varphi_t(x)\leq \varphi_t(y)$ whenever $x\leq y$ and $t\geq 0$ and \textit{stongly monotone} if $\varphi_t(x)\ll \varphi_t(y)$ whenever $x\leq y$, $x\neq y$ and $t>0$ (see \cite{HS05} and \cite{S95}).

The following Proposition shows that in $\mathbb{R}^n$, a monotone system is differentially positive (see \cite[Theorem 1]{ForniandSepulchre16}).

\begin{prop}\label{system equivalence}
	Let $M$ be the $n$-dimensional Eucliead space and $C$ is a convex cone in $M$ such that $C$ forms a constant cone field $C_M$. Then a system is monotone with respect to partial order ``$\leq$" generated by cone $C$ if and only if this system is differentially positive with respect to cone field $C_M$.
\end{prop}
\begin{proof}
	By Proposition \ref{order equivalence}, order ``$\leq$" and order ``$\leq_M$" are equivalent.
	
	Suppose that the system is differentially positive. For $x\leq y$, there exists a conal curve $\gamma : [0,1]\to M$ such that $\gamma(0)=x$, $\gamma(1)=y$ and $\frac{d}{ds}\gamma(s)\in C$. Then, we obtain that $\varphi_t(\gamma(\cdot))$ is also a conal curve for each $t\geq 0$. In fact, $\frac{d}{ds}\varphi_t(\gamma(s))=d\varphi_t(\gamma(s))\frac{d}{ds}\gamma(s)$. Since $\varphi_t$ is differentially positive and $\frac{d}{ds}\gamma(s)\in C$, then $\frac{d}{ds}\varphi_t(\gamma(s))\in C$. So, $\varphi_t(x)\leq \varphi_t(y)$ for $t\geq 0$. And hence, $\varphi_t$ is monotone with respect to ``$\leq$".
	
	If the system is monotone. For $x\in M$ and $v\in C$, then there is a conal curve $\gamma(s)$, $s\in I$, such that $\gamma(0)=x$, $\frac{d}{ds}\gamma(s)\vert_{s=0}=v$ and $x\leq \gamma(s)$ with $s\geq 0$. Since $\varphi_t$ is monotone, then for each $t\geq 0$, $\varphi_t(x)\leq \varphi_t(\gamma(s))$, i.e., $\varphi_t(\gamma(s))-\varphi_t(\gamma(0))\in C$. Thus, $\frac{d}{ds}\varphi_t(\gamma(s))\vert_{s=0}=\lim\limits_{\Delta s\to 0^+}\frac{\varphi_t(\gamma(\Delta s))-\varphi_t(\gamma(0))}{\Delta s}\in C$. So, we obtain that $d\varphi_t(x)v=\frac{d}{ds}\varphi_t(\gamma(s))\vert_{s=0}\in C$. And hence, the system is differentially positive.
\end{proof}


\end{document}